\documentclass[a4paper,10pt]{article}
\usepackage[english]{babel}
\usepackage[latin1,utf8]{inputenc}
\usepackage[T1]{fontenc}
\usepackage{amsfonts}
\usepackage{amsmath}
\usepackage{amsthm}
\usepackage{amssymb}
\usepackage{fancyhdr}
\usepackage{indentfirst}
\usepackage{color}
\usepackage{graphicx}
\usepackage{newlfont}
\usepackage{mathrsfs}
\usepackage{latexsym}
\usepackage{wrapfig}
\usepackage{subfig}
\usepackage{braket}
\usepackage{hyperref}
\usepackage{pdfsync}

\hypersetup{colorlinks=true,linkcolor=blue}

\theoremstyle{plain}
\newtheorem{thm}{Theorem}
\newtheorem{cor}{Corollary}
\newtheorem{lem}{Lemma}
\newtheorem{prop}{Proposition}

\theoremstyle{definition}
\newtheorem{defn}{Definition}
\theoremstyle{remark}

\theoremstyle{remark}
\newtheorem{oss}{\textbf{Remark}}
\newcommand{\R}{\mathbb{R}}
\newcommand{\Z}{\mathbb{Z}}
\newcommand{\N}{\mathbb{N}}
\newcommand{\1}{\sqrt{\frac{\vert \log \varepsilon \vert}{\varepsilon}}}

\newcommand{\3}{\varepsilon\vert \log \varepsilon \vert}
\def\e{\varepsilon}

\begin{document}
\title{\textbf{Static, Quasi-static and Dynamic Analysis for Scaled Perona-Malik Functionals}}
\author{
\scshape{ Andrea Braides}\\
Università di Roma ``Tor Vergata''\\
\texttt{braides@mat.uniroma2.it}\\
\\
\scshape{Valerio Vallocchia} \\
Università di Roma ``Tor Vergata''\\
\texttt{vallocch@mat.uniroma2.it}\\}
\date{}
\maketitle
\begin{abstract}
We present an asymptotic description of local minimization problems, and of quasi-static and dynamic evolutions of  scaled Perona-Malik functionals.
The scaling is chosen such that these energies $\Gamma$-converge to the Mumford-Shah functional by a result by Morini and Negri \cite{otto}. 
\end{abstract}

 \section{Introduction}
 The Perona-Malik {\em anisotropic diffusion} technique in Image Processing 
 \cite{PM} is formally based on a gradient-flow dynamics related to the non-convex energy
 \begin{equation}\label{PM-cont}
F_{{\scriptstyle P\!M}}(u)=\int_\Omega \log \Bigl(1+{1\over K^2}|{\nabla u}|^2\Bigr)\,dx,
 \end{equation}
 where $u$ represents the output signal or picture defined on $\Omega$ and $K$ a tuning parameter. 
 In the convexity domain of the energy function; i.e., if $|\nabla u|\le K$ the effect of the gradient 
 flow is supposed to smoothen the initial data, while on discontinuity sets where $|\nabla u|=+\infty$
 the gradient of the energy is formally zero and no motion is expected. In reality, such a gradient flow
 is ill-posed and even in dimension one we may have strong solutions only for some certain classes of initial data, 
 or weak solutions which develop complex microstructure. However, the anisotropic diffusion technique is always applied
 in a discrete or semi-discrete context, where energy $F_{{\scriptstyle P\!M}}$ is only a (formal)
 continuum approximation of some discrete energy defined on some space of finite elements or
 in a finite-difference context. Indeed, it is well known that convex-concave energies, which give
 ill-posed problems if simply extended from a discrete to a continuum context, are related to
 well-posed problems in a properly defined passage discrete-to-continuum. In a static framework, the 
 prototype of this argument dates back to the analysis of Chambolle \cite{Ch}, who showed that the 
 Blake-Zisserman weak-membrane discrete energy (involving truncated quadratic potentials) \cite{BZ} 
can be approximated by the Mumford-Shah functional \cite{MS}. The latter functional (together with its anisotropic variants) is a common continuum approximation of a class of lattice energies with convex-concave
energy functions, which also comprises atomistic energies such as Lennard-Jones ones \cite{BLO}
and the discretized version of the Perona-Malik functional itself as showed by Morini and Negri \cite{otto}
(see also \cite{tre} Section 11.5). The approximation of these lattice energies is performed by
considering the lattice spacing $\e$ as a small parameter and suitably scaling the energies. 
In the case of the Perona-Malik discretized energy on the cubic lattice $\e\Z^n$ the scaled functionals
$$
F_{\varepsilon}(u):=\sum_{i,j} \frac{\e^{n-1}}{|\log \varepsilon |} \log\left(1+|\log \varepsilon |\frac{
|u_i-u_j|^2}{\varepsilon}\right)
$$
($u_i$ denotes the value of $u$ at $i\in\e\Z^n$ and the sum is performed on nearest neighbours)
$\Gamma$-converge to a Mumford-Shah energy, with an anisotropic surface energy density in
dimension larger than one \cite{otto,tre}. This means that the solutions to global minimization 
problems involving $F_\e$, identified with their piecewise-constant interpolations, converge as
$\e$ tends to zero to the solutions to the corresponding global minimization problems involving
the Mumford-Shah functional. Examples of such global minimum problems comprise problems 
in Image Processing with an additional lower-order fidelity term (typically an $L^2$-distance 
of $u$ from the input datum $u_0$).

In this paper we analyze how much this approximation procedure can be extended beyond the global-minimization standpoint by examining the one-dimensional case. 
It is known that $\Gamma$-convergence cannot be easily extended as a theory 
to the analysis of the behaviour of local minima or to a dynamical setting beyond, essentially,
the ``trivial'' case of convex energies \cite{due,BCGS}. However, several recent examples suggest that 
for problems with concentration some quasistatic and dynamic models are compatible with 
$\Gamma$-convergence (such as for Ginzburg-Landau \cite{SS} or for Lennard-Jones \cite{cinque}
energies). We show that this holds also for one-dimensional quasistatic and dynamic problems obtained 
as minimizing movements along the family $F_\e$ \cite{due}. They coincide with the corresponding 
problems related to the one-dimensional Mumford-Shah functional  
\begin{align}
\label{ms}
M_s(u)=\int_0^1 |u'|^2+\#(S(u))
\end{align}
defined on piecewise $H^1$-function, where $S(u)$ is the set of discontinuity points of $u$. 
We note the difficulty of the extension to dimension larger or 
equal than two, for which a characterization of minimizing movements for the Mumford-Shah functional is still lacking \cite{AB}. 

In the quasistatic case, our analysis relies on a modeling assumption, that amounts to considering as dissipated the energy beyond the convexity threshold. Again, we show that the Mumford-Shah energy is an approximation of $F_\e$ also in that framework.
For an analysis of the quasistatic case in dimension larger than one within its application to Fracture Mechanics we refer to \cite{BFM}.

When local minimization is taken into account, then we show that indeed for some classes of problems the pattern of local minima of $F_\e$ differ from that of $M_s$. The computation of the $\Gamma$-limit can nevertheless be used as a starting point for the construction of ``equivalent theories'', which keep the simplified form of the $\Gamma$-limit but maintain the pattern of local minima. This process has been 
formalized in \cite{BT}. In our case we prove the $\Gamma$-equivalence of energies of the form
\begin{equation}
\label{mse}
G_\e(u)=\int_0^1 |u'|^2dx+\sum_{x\in S(u)} \frac{1}{\vert \log\varepsilon \vert}g\left(\1 |u^+-u^-|\right)
\end{equation}
with $g$ a concave function with $g'(0)=1$ and $g(w)\sim 2\log w$ for $w$ large.
Another case in which the Mumford-Shah functional is not a good approximation of the scaled Perona-Malik is for the long-time behaviour of gradient-flow dynamics, as we briefly illustrate in the final section.

% 
% TO DO...
% We recall  an important result of Morini and Negri (\cite{otto}), where it is  proved that the scaled Perona-Malik functional\eqref{pm}  $\Gamma$-converges to Mumford-Shah functional:
%Here we denote with $S(u)$  the jump set of $u$ .
\section{The scaled Perona-Malik functional}
We consider a one-dimensional system of $N$ sites with nearest-neighbour interactions. 
Let  $\varepsilon=1/N$ denote the {\em spacing parameter} and  let $u:=\left(u_0,\dots,u_{N}\right)$ be a  function defined on the lattice $I_{\varepsilon}=\varepsilon\mathbb{Z}\cap[0,1]$. If taking $\e$ as parameter, we also denote $N=N_\e$.

We define the {\em scaled one-dimensional Perona-Malik} functional as 
\begin{align}
\label{pm}
F_{\varepsilon}(u):=\sum_{i=1}^{{N_\varepsilon}} \frac{1}{|\log \varepsilon |} \log\left(1+|\log \varepsilon |\frac{
|u_i-u_{i-1}|^2}{\varepsilon}\right).\end{align}

The behaviour of global minimum problems involving $F_\e$ as $\e\to 0$ can be described through 
the computation of their $\Gamma$-limit. To that end, we define the {\em discrete-to-continuum convergence} $u_\e\to u$ as the $L^1$-convergence of the piecewise-constant interpolations $u_\e(x)= (u_\e)_{\lfloor x/\e\rfloor}$ to $u$.

\begin{thm}[Morini and Negri \cite{otto}]\label{MNth}
The domain of the $\Gamma$-limit of the functionals $F_\e$ as $\e\to 0$ is the space of piecewise-$H^1$ functions on which it coincides with the Mumford-Shah functional $M_s$ defined in {\em(\ref{ms})}.
\end{thm}

With the application of the Mumford-Shah functional to Fracture Mechanics in mind, 
by this result the Perona-Malik energy can be interpreted in terms of a mass-spring model 
approximation of  Griffith brittle-fracture theory. We will then refer in what follows to the 
quantities $u_i-u_{i-1}$ (or $w_i$ in the notation introduced below) as ``spring elongations''.

\smallskip

As a consequence of Theorem \ref{MNth} we easily deduce that minimum problems of the form
$$
\min\Bigl\{F_\e(u)+\alpha\sum_{i=0}^{N_\e}\e|u_i-u^0_i|^2\Bigr\}
$$
converge (both as minimum value and minimizers are concerned) to the minimum problem
$$
\min\Bigl\{M_s(u)+\alpha\int_0^1|u-u^0|^2\,dx\Bigr\},
$$
provided, e.g.,  that $u^0\in L^\infty(0,1)$ is such that the interpolations $\{u^0_i\}$ converge to $u^0$ \cite{otto}.

\smallskip
The heuristic explanation of why the scaling in (\ref{pm}) gives the Mumford-Shah functional is as follows. If the difference quotient $\e(u_i-u_{i-1})$ is bounded
then 
$$
|\log \varepsilon |\frac{
|u_i-u_{i-1}|^2}{\varepsilon}<\!<1
$$
so that 
$$
\frac{1}{|\log \varepsilon |} \log\left(1+|\log \varepsilon |\frac{
|u_i-u_{i-1}|^2}{\varepsilon}\right)\sim
\frac{
|u_i-u_{i-1}|^2}{\varepsilon}= \e \Bigl|\frac{u_i-u_{i-1}}{\varepsilon}\Bigr|^2,
$$
which gives a discretization of the Dirichlet integral. Conversely, if at an index
$i$ we have $|u_i-u_{i-1}|^2\sim c>0$ (corresponding to a jump point in the limit)
then 
$$
\frac{1}{|\log \varepsilon |} \log\left(1+|\log \varepsilon |\frac{
|u_i-u_{i-1}|^2}{\varepsilon}\right)\sim
\frac{1}{|\log \varepsilon |} \log\left(1+|\log \varepsilon |\frac{
c}{\varepsilon}\right)\sim 1.$$
The actual proof of Theorem \ref{MNth} is technically complex since the analysis of the two 
possible behaviours of discrete functions (as Dirichlet integral or as jump points) does not 
correspond exactly to examining the difference quotients above or below the inflection 
points (contrary to what can be done for truncated quadratic potentials \cite{Ch}).

%\begin{figure}[h!]
%\center
%\begin{tikzpicture}
%\begin{axis} [axis lines=middle, xmin=-5, xmax=5, enlargelimits, xlabel=$z$, ylabel=$J(z)$,ytick={},yticklabels={},xtick={1,-1}, xticklabels ={$1$,$-1$},height= 6 cm]
%\addplot [samples=200, thick]  {ln(1+x^2)};
%\addplot [dashed] coordinates{(1,0) (1,0.69314718056)};
%\addplot [dashed] coordinates{(-1,0) (-1,0.69314718056)};
%\end{axis}
%\end{tikzpicture}
%\caption{$J(z)=\log(1+z^2)$}
%\label{f1}
%\end{figure}

\bigskip
\begin{figure}[h!]
\centerline{\includegraphics [width=2.5in]{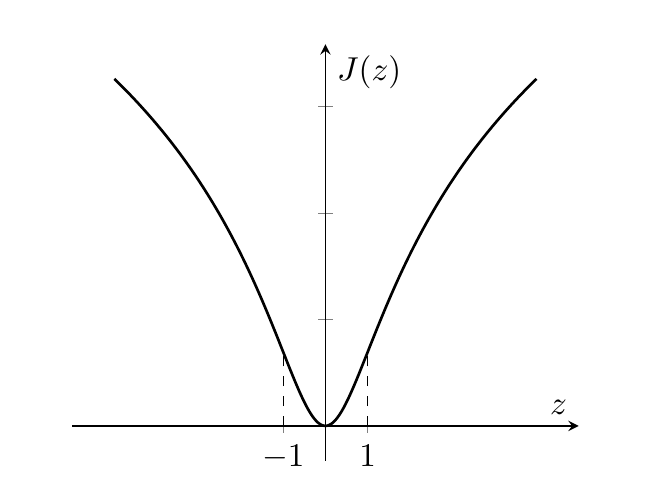} \includegraphics [width=3in]{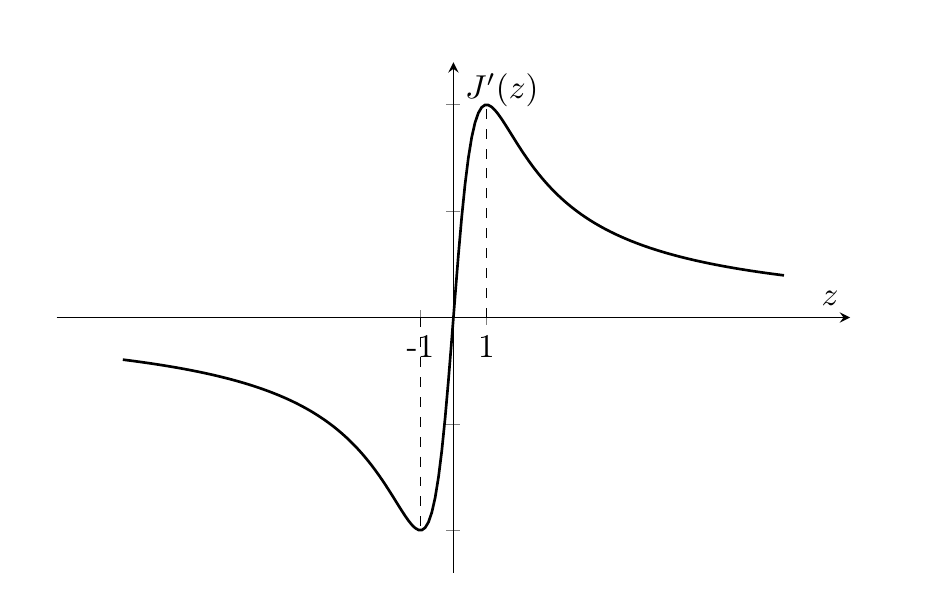}\qquad}
\caption{The potential $J$ and its derivative}\label{f1}
   \end{figure}
We find it convenient to rewrite \eqref{pm} in terms of the function 
$$J(z)=\log(1+z^2)$$ (see Fig.~\ref{f1}). 
The energy then reads
\begin{align}
\label{pm2}
&F_{\varepsilon}(u)=\sum_{i=1}^{N_{\varepsilon}} \frac{1}{|\log \varepsilon |} J\Bigl(\sqrt{\frac{|\log \varepsilon |}{\varepsilon}}(u_i-u_{i-1})\Bigr).
\end{align}

Note that the function $J$ satisfies:

$\bullet$ it is an even function;

$\bullet$ it is monotone increasing in $[0,+\infty)$ and monotone decreasing in $(-\infty,0]$;

$\bullet$ it is convex  in the interval $[-1,1]$ and is concave on $[1,+\infty)$;

$\bullet$ there exists a constant $c>0$ such that $J(z)\geq c\,z^2$ for all $|z|\leq 1$.

\section{Analysis of local minima}
$\Gamma$-convergence does not describe the behaviour of local minimum problems.
In this section we compute energies defined on piecewise-$H^1$ functions 
which are close to $F_\e$ in the sense of $\Gamma$-convergence 
and maintain the pattern of local minima.

\subsection{Local minima for $F_\e$ with Dirichlet boundary conditions}
We first characterize local minima of \eqref{pm} with Dirichlet  
boundary conditions $u_0=0$ and $u_N=\lambda$ for  $\lambda \in \mathbb{R}$.

\smallskip
The stationarity conditions for \eqref{pm2} read
\begin{align}
\label{stationary}
J'\left(\sqrt{\frac{|\log \varepsilon |}{\varepsilon}}
(u_i-u_{i-1})\right)=\sigma
\end{align}for some $\sigma\in\R$.

From now on we simplify the notation introducing the scaled variable
\begin{equation}\label{wii}
w_i:=\sqrt{\frac{|\log \varepsilon |}{\varepsilon}}(u_i-u_{i-1}).
\end{equation}
First, we observe that we can limit our analysis to the case  $w_i\geq 0$ for all $i=1,\dots,N$. 
In fact, $J(z)$ is even; moreover, for $w_i\geq 0$  also  $J'(w_i)\geq 0$, 
so that  if $w_i\geq 0$ for some $i$ then by \eqref{stationary} all the remaining $w_j$ must be non-negative. 
In the same way we can handle the case with $w_i<0$.
%\begin{figure}[h!]
%\center
%\begin{tikzpicture}
%\begin{axis} [axis lines=middle, xmin=-10, xmax=10,enlargelimits, xlabel=$z$, ylabel=$J'(z)$,xtick={1,-1}, xticklabels ={1,-1}, ytick={},yticklabels={},width=10 cm, height=7 cm]
%\addplot [domain=-10:10, samples=200, thick]  {2*x /(1+x^2)};
%\addplot [dashed] coordinates{(1,0) (1,1)};
%\addplot [dashed] coordinates{(-1,0) (-1,-1)};
%\end{axis}
%\end{tikzpicture}
%\caption{$J'(z)=\dfrac{2z}{1+z^2}$}
%\label{f2}
%\end{figure}

Now we characterize local minimizers in some different cases.

$\bullet$ Suppose that $w_i< 1$ for all $i$. Then the monotonicity of $J'$ in $[0, 1]$ implies  $w_i=w_j$ for all $i,j=1,\dots, N$. In particular, thanks to the boundary conditions, we have that 
\begin{align}
\frac{w_i}{\sqrt{\varepsilon |\log \varepsilon|} }=\lambda\ \ \ \ \ \Longrightarrow \ \ \ \ \ \lambda< \frac{1}{\sqrt{\varepsilon |\log \varepsilon|}}.\end{align}
This is a local minimum. In fact, for $w_i< 1$ the function is strictly convex, which means that $J''(w_i)>0$ for all $i$. 
We observe  that when $\lambda=0$ the only solution is the trivial one ($w_i\equiv 0$), which is a global minimum.

$\bullet$
 Not more than one index can satisfy $w_i>1$. Indeed, suppose that there exist two indices such that $w_1=w_2(=\overline{w})>1$ and consider a perturbed configuration $(v_0,\dots,v_N)$ such that, denoting \\
\begin{align*}
\tilde{v}_i:=(v_i-v_{i-1})\sqrt{\frac{|\log \varepsilon |}{\varepsilon}},
\end{align*}
we have \\
- $\tilde{v}_i=w_i \ \mbox{for all} \ i\ne 1,2$\\
- $\tilde{v}_1=\overline{w}+s$ and $\tilde{v}_2=\overline{w}-s$.\\
We observe that the difference between the energies of the  configurations is
\begin{align*}
f(s)=J(\overline{w}+s)+J(\overline{w}-s)-2J(\overline{w})
\end{align*}
and it is such that $f'(0)=0$ and $f''(0)=2J''(\overline{w})<0$.
This means that we have  a local maximum in $s=0$, so that we cannot have a local minimum for such configuration.

$\bullet$
 In  case $w_i=1$ for some $i$ then we have $w_j=1$ for all $j$, with the consequence  that $\lambda=\frac{1}{\sqrt{\varepsilon |\log \varepsilon|}}$. This is not a local minimum. Indeed, consider the perturbed energy
\begin{align}
f(t)=(N-1)J\left(1-\frac{t}{N-1}\right)+J(1+t).
\end{align}
Then we observe that
\begin{align*}
f'(0)=0,\qquad
f''(0)=0,\qquad
f'''(0)=J'''(1)\left(1-\dfrac{1}{(N-1)^2}\right)<0,
\end{align*}
so that $0$ is not a minimum for $f$.

$\bullet$
 Finally, we take into account the case with only one index exceeding the convexity threshold. Suppose that there exists an index $i$ such that 
$$\sqrt{\frac{\vert \log \varepsilon \vert}{\varepsilon}}(u_i-u_{i-1})=\sqrt{\frac{\vert \log \varepsilon \vert}{\varepsilon}}w>1.$$
Then, thanks to the boundary conditions, we can rewrite the energy of the system as
\begin{align*}
\tilde{F}_{\varepsilon}(w)=(N-1) J\left(\sqrt{\frac{\vert \log \varepsilon \vert}{\varepsilon}}\left(\frac{\lambda-w}{N-1}\right)\right)+ J\left(\sqrt{\frac{\vert \log \varepsilon \vert}{\varepsilon}} w\right),
\end{align*}
defined on the domain 
\begin{align}
\label{cond}
A=\Biggl\{ w\in \R^+ : w\geq \max \left\{\sqrt{\frac{\varepsilon}{\vert \log \varepsilon \vert}} , \lambda-(N-1)\sqrt{\frac{\varepsilon}{\vert \log \varepsilon \vert}}\right\}\Biggr\}.
\end{align}

To compute the values of $w$ we impose that $\tilde{F}_{\varepsilon}'(w)=0$. This gives
\begin{align*}
J'\left(\sqrt{\frac{\vert \log \varepsilon \vert}{\varepsilon}}\left(\frac{\lambda-w}{N-1}\right)\right)- J'\left(\sqrt{\frac{\vert \log \varepsilon \vert}{\varepsilon}} w\right)=0.
\end{align*}
We then obtain three solutions: $w=\varepsilon\lambda$ (and hence $w_i\equiv w_0$, which is not a local minimum) and
\begin{align}
\label{sol}
w_{1,2}&=\frac{\lambda}{2}\pm \sqrt{\frac{\lambda^2}{4}-\frac{1-\varepsilon}{\vert \log \varepsilon \vert}}.\end{align}
We observe that the solutions in \eqref{sol} are both positive for $\lambda >2\sqrt{\frac{1-\varepsilon}{\vert \log \varepsilon \vert}}$ but we need to check for which $\lambda$ they  belong to $A$. We get that
\begin{align}
w_1&=\frac{\lambda}{2}+ \sqrt{\frac{\lambda^2}{4}-\frac{1-\varepsilon}{\vert \log \varepsilon \vert}}\in A \qquad \Longleftrightarrow\qquad \lambda\geq 2\sqrt{\frac{1-\varepsilon}{\vert\log\varepsilon\vert}}\\
w_2&=\frac{\lambda}{2}- \sqrt{\frac{\lambda^2}{4}-\frac{1-\varepsilon}{\vert \log \varepsilon \vert}} \in A\qquad\Longleftrightarrow\qquad 2\sqrt{\frac{1-\varepsilon}{\vert \log \varepsilon \vert}}\leq\lambda\leq \frac{1}{\sqrt{\vert\log\varepsilon\vert}}.
\end{align}
Since we are interested in local minima, we have to verify that $\tilde{F}_{\varepsilon}''(w_i)>0$, which means
\begin{align*}
\tilde{F}_{\varepsilon}''(w_i)=\frac{1}{(N-1)} J''\left(\1\left (\frac{\lambda-w_i}{N-1}\right)\right)+J''\left(\1 w_i\right)>0.
\end{align*}
Simple computations show that
\begin{align*}
\tilde{F}''_{\varepsilon}(w_1)&>0 \Longleftrightarrow \lambda > 2 \sqrt{\frac{1-\varepsilon}{\vert \log \varepsilon \vert}}\\
\tilde{F}''_{\varepsilon}(w_2)&<0 
\end{align*}
So, we can finally  state that, when an index exceeds the convexity threshold, there exist only a local minimum for $\lambda > 2\sqrt{\frac{1-\varepsilon}{\vert\log\varepsilon\vert}}$.

%\begin{minipage}{0.5\linewidth}
%\begin{tikzpicture}
%\begin{axis} [axis lines=middle, height=6cm,width=7cm,xmin=-4, xmax=4,ymin=0,ymax=5,  xlabel=$\lambda$, ylabel=$F_{\varepsilon}(u)$,ytick={},yticklabels={},xtick={2,-2}, xticklabels ={$\frac{1}{\sqrt{\varepsilon\vert \log\varepsilon\vert}}$,$-\frac{1}{\sqrt{\varepsilon\vert\log\varepsilon\vert}}$}]
%\addplot [domain=-2:2,samples=200, thick]  {x^2};
%\addplot [black,mark=otimes] coordinates {(2,4)};
%\addplot [black,mark=otimes] coordinates {(-2,4)};
%\addplot [solid] coordinates {(0.5,1) (4,1)};
%\addplot [solid] coordinates {(-0.5,1) (-4,1)};
%\addplot [dashed] coordinates{(2,0) (2,4)};
%\addplot [dashed] coordinates{(-2,0) (-2,4)};
%\end{axis}
%\end{tikzpicture}
%\end{minipage}
%\begin{minipage}{0.5\linewidth}
%\begin{tikzpicture}
%\begin{axis} [axis lines=middle, height=6cm,width=7cm,xmin=-4, xmax=4,ymin=0,ymax=5,  xlabel=$\lambda$, ylabel=$M_s(u)$,ytick={},yticklabels={},xtick={}, xticklabels={}]
%\addplot [domain=-2.5:2.5,samples=200, thick]  {x^2};
%\addplot [solid] coordinates {(-4,1) (4,1)};
%\addplot [solid] coordinates {(-4,2) (4,2)};
%\addplot [solid] coordinates {(-4,3) (4,3)};
%\addplot [solid] coordinates {(-4,4) (4,4)};
%\end{axis}
%\end{tikzpicture}\\
%\end{minipage}

\begin{figure}[h!]
\centerline{\includegraphics [width=4.5in]{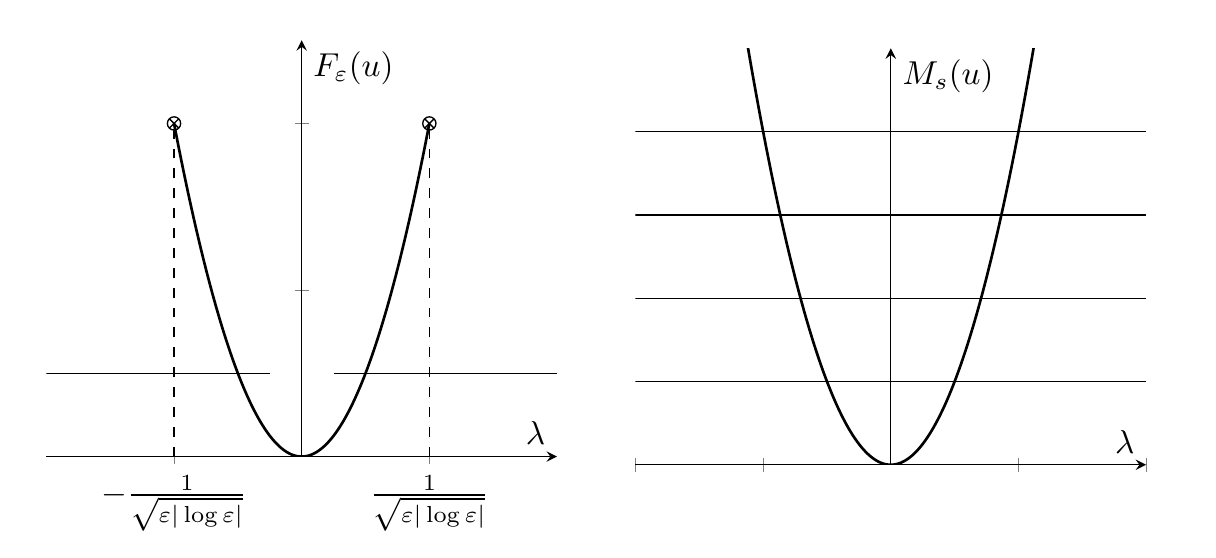}\qquad}
\caption{Perona-Malik and Mumford-Shah local minima}\label{Fig3}
   \end{figure}

\begin{oss}[local and global minima for the Mumford-Shah functional]
(a) We note that the pattern of local minima for the functional $M_s$ subjected to 
the boundary conditions $u(0-)=0$ and $u(1+)=\lambda$ differs from that of
$F_\e$. Indeed, we have 

$\bullet$ the function $u_\lambda(x)=\lambda x$, corresponding to the energy 
$M_s(u_\lambda)=\lambda^2$;

$\bullet$ all functions $u$ with $u'=0$, for which we have $M_s(u)=\#(S(u))$.

Note that for $\lambda=0$ we cannot have a local minimum of the second type 
with $M_s(u)=1$, while for all other $\lambda$ we have no restriction on the number
of jumps. A description of the energy of local minima in dependence of $\lambda$ 
and compared with those of Perona-Malik energies is pictured in Fig.~\ref{Fig3}.

\smallskip
(b) from the analysis above we trivially have that the (global) minimum energy in dependence
of the boundary datum $\lambda$ is $\min\{\lambda^2,1\}$, achieved on the linear function
$u_\lambda$ if $|\lambda|\le 1$ and on any function $u(x)=\lambda\chi_{[x_0,+\infty)}$ jumping
in $x_0$ if $|\lambda|\ge 1$.
\end{oss}

\subsection{$\Gamma$-equivalence}\label{gammaec}
In this section we propose a ``correction'' to the $\Gamma$-limit of $F_\e$.
In place of $M_s$ we want to compute functions $G_\e$ such that

$\bullet$ $G_\e$ maintain the structure of $M_s$; i.e., they are defined on piecewise-$H^1$ functions
and can be written as the sum of the Dirichlet integral and an energy defined on the jump set $S(u)$;

$\bullet$ the structure of local minima of $G_\e$ is the same as that of $F_\e$;

$\bullet$ $G_\e$ and $F_\e$ are ``equivalent'' with respect to $\Gamma$-convergence.

\smallskip
We recall the general definition of $\Gamma$-equivalence \cite{due}.

\begin{defn} Let $\left\{F_{\varepsilon}\right\}$ and $\left\{G_{\varepsilon}\right\}$ be sequences of functionals on a separable metric space $X$. We say that they are equivalent by $\Gamma$-convergence (or $\Gamma$-equivalent) if there exists a sequence $\left\{m_{\varepsilon}\right\}$ of real numbers such that, if $\left\{F_{\varepsilon_j}-m_{\varepsilon_j}\right\}$ and $\left\{G_{\varepsilon_j} -m_{\varepsilon_j} \right\}$ are $\Gamma$-converging sequences, their $\Gamma$-limits coincide and are proper (i.e., not identically $+\infty$ and not taking the value $-\infty$).
\end{defn}
In our case this definition simplifies: we may consider $m_{{\varepsilon}}\equiv 0$ and we look for functionals  which $\Gamma$-converge to $M_s$. We look for $G_\e$ of the form
\begin{align}
G_{\varepsilon}(u)=\int_0^1 |u'|^2 dt+\sum_{t\in S(u)}g_{\varepsilon}(|u^+-u^-|)
\end{align} 
with  boundary conditions $u^-(0)=0$, $u^+(1)=\lambda$.
 Furthermore, the scaling argument in the definition of  $F_\e$ 
suggests to look for $g_{\varepsilon}$ with the same scaling of the Perona-Malik functional, so that $$g_{\varepsilon}(z)=\frac{1}{\vert \log\varepsilon \vert}g\left(\1\, z\right).$$

In order to have functionals with the same structure of local minimizers as Perona-Malik's, we have to require  that, if $u$ is a minimizer, then $\#(S(u))\leq 1$. This condition surely holds when $g$ is concave:  in fact if $z_1$,$z_2$ are two points in $S(u)$, then the function $t\to g(w_1+t)+g(w_2-t)$ is still concave.
To ensure the $\Gamma$-convergence of $G_{\varepsilon}$ to $M_s$ we impose that
\begin{align*}
\lim_{\varepsilon\to 0}\frac{1}{\vert \log\varepsilon \vert}g\left(\1 z\right)=1.
\end{align*}
This is ensured by the condition 
$$
\lim_{w\to+\infty}{g(w)\over 2\log(w)}=1.
$$

Finally we require that $g(0)=0$ and that the slope of $g_{\varepsilon}$ in the origin be ${1}/{\sqrt{\varepsilon\vert \log \varepsilon \vert}}$, which is the value after which it is energetically convenient to introduce a fracture (further details for this argument can be found in \cite{quattro}). This means that
\begin{align*}
\left.\frac{1}{\vert \log \varepsilon \vert}g'\left(\1 w\right)\right|_{w=0}&=\dfrac{1}{\sqrt{\varepsilon\vert \log \varepsilon \vert}},
\end{align*}
which gives $ g'(0)=1$.

Summarizing, we have checked that the functionals $G_\e$ above are $\Gamma$-equivalent to $F_\e$
and maintain the same patter of local minima, provided that

$\bullet$ $g:[0,+\infty)\to[0,+\infty)$ is concave;

$\bullet$ $g(0)=0$, $g'(0)=1$ and $\displaystyle
\lim_{w\to+\infty}{g(w)\over 2\log(w)}=1$.

\section{Quasi-static motion}
In this section we compare quasistatic motion (also sometimes denoted as variational evolution) 
for $F_\e$ with that of the Mumford-Shah functional. We adopt as the definition of quasistatic motion
that of a limit of equilibrium problems involving energy and dissipation 
with varying boundary conditions. For more general definitions
and related discussion we refer to \cite{MR,BFM,due}. 

We consider a sufficiently regular function $h:[0,+\infty)\to\R$ with $h(0)=0$ and boundary conditions $u_0=0$, $u_N=h(t)$, which means that the function $h$ is describing the position of the endpoint of the $N$-th spring.

\begin{oss}[quasistatic motion of the Mumford-Shah functional with increasing fracture]\label{quaMS}
In the framework of Fracture Mechanics, the Dirichlet integral in the Mumford-Shah functional  is interpreted as an elastic energy and the jump term as a dissipation term necessary to create a crack. The {\em dissipation principle} underlying crack motion is that, once a crack is created, this jump term cannot be recovered. If we apply time-dependent boundary conditions $u(0,t)=0$ and $u(1,t)= h(t)$ then a solution is given by 
$$
u(x,t)= \begin{cases} 
h(t)x & \hbox{ if $t\le t_h$}\cr
h(t)\chi_{[x_0,1]}& \hbox{ if $t> t_h$,}
\end{cases}
$$
where $t_h=\inf\{t: h(t)>1\}$ and $x_0$ in any given point in $[0,1]$. 
With this definition the crack site $K(t)=\bigcup_{s\le t} S(u(\cdot,s))$ is 
non-decreasing with $t$ and $u(\cdot,t)$ is a global minimizer of the
Mumford-Shah energy on $(0,1)\setminus K(t)$. 
\end{oss}

In the case of the Perona-Malik functional we do not have a distinction between
elastic and fracture parts of the energy. 
We then assume the following dissipation principle, where those two parts
are identified with the convex and concave regions of the energy function, respectively.\\

\textbf{Dissipation Principle}: if for some $i$ the spring elongation $w_i$ in (\ref{wii}) 
overcomes the convexity threshold then the energy $J(w_i)$ cannot be recovered.\\

In analogy with the case of the quasistatic motion of the Mumford-Shah functional, this 
principle will be translated into modifying the total energy on indices $i$ for which 
$w_i$ have overcome the convexity threshold during the evolution process.

\smallskip
As mentioned above, the quasistatic motion of $F_\e$ can be defined through a time-discrete approximation
as follows. The analogous procedure for $M_s$ produces the solutions for the quasistatic motion of $M_s$ 
as in Remark \ref{quaMS}. 

We fix a {\em time step} $\tau>0$, and for all $k\in\N$ we consider the boundary conditions
$u_0=0$ and $u_N=h(k\tau)$, and the related ``time-parameterized'' minimum problems subjected 
to the Dissipation Principle stated above. We now analyze the properties of the corresponding
solutions. For simplicity of notation we will consider the case when $h\ge 0$.

Note preliminarily that by the convergence of the global minima of $F_\e$ to those of $M_s$ subjected to
Dirichlet boundary conditions, we deduce the existence of a threshold $\tilde{h}_{\varepsilon}$ beyond which 
it is energetically convenient that one elongation $w_i$ lies in the concavity domain of $J$. Note that $\tilde{h}_{\varepsilon}\to 1$ as $\e\to 0$.

%Like in the static analysis we will focus on the case $w_i\geq 0$ and then we will extend our results by symmetry.\\ 

We can point out different behaviours as follows

$\bullet$ If for all $k'<k$  we have $|h(k'\tau)|\le\tilde{h}_{\varepsilon}$
then the dissipation principle is not enforced
and the corresponding solution $u^k$ is the only minimizer for the energy \eqref{pm2} corresponding to
the interpolation of the linear function $h(k\tau)x$. Its energy is
\begin{align}
F_{\varepsilon}(u^k)=\frac{1}{\e\vert\log\varepsilon\vert}J\left(h(k\tau)\sqrt{\3}\right).
\end{align}

$\bullet$ Now suppose that at some $k'$ we have $|h(k'\tau)|>\tilde{h}_{\varepsilon}$.  
It is not restrictive to suppose that $h(k'\tau)>0$ (the negative case being treated symmetrically). Then,
there exists an index $i$ such that the $w_i$ corresponding to the solution exceeds $\tilde{h}_{\varepsilon}$. Without losing generality we will suppose $i=N$.

We have two possibilities.

1. If $h((k'+1)\tau)> h(k'\tau)$ then we have to minimize 
\begin{align*}
(N-1) J\Bigr(\1\Bigl({u_{N-1}\over N-1}\Bigr)\Bigr)+J\Bigr(\1(u_N-u_{N-1})\Bigr) 
\end{align*}
under the boundary conditions  $u_0=0$, $u_N=h(k\tau)$. Note that we have used the convexity of $J$
to simplify the contribution of the first $N-1$ interactions. Indeed, in this case
their common elongation is 
$$
u_i-u_{i-1}= {u_{N-1}-u_0\over N-1} = {u_{N-1}\over N-1}.
$$

The previous considerations  show that there exists a unique minimizer for $h(k\tau)>2\sqrt{(1-\varepsilon)/\vert\log\varepsilon\vert}$ and, denoted 
\begin{align*}
w:=\frac{h(k\tau)}{2}+\sqrt{\frac{h(k\tau)^2}{4}-\frac{1-\varepsilon}{\vert\log\varepsilon\vert}},
\end{align*}
the energy reads
\begin{align*}
F_{\varepsilon}(k)=\frac{1}{\vert\log\varepsilon\vert}\Bigl((N-1) J\Bigr(\1 \Bigr(\frac{h(k\tau)-w}{N-1}\Bigr)\Bigr)+J\Bigr(\1w\Bigr)\Bigr). 
\end{align*}
We can iterate this process as long as $k\mapsto h(k\tau)$ increases.
Note that in this case the application of the Dissipation Principle does not change 
the minimum problems since all $J(w_j)$ are increasing with $k$;

2. The function $k\mapsto h(\tau k)$ has a local maximum at $\bar{k}$. 
In this case the Dissipation Principle does force a change in the 
minimization problem.
As long as $h(\tau k)\le h(\tau \bar k)$ we have to minimize
\begin{align}
\label{probl}
(N-1)&J\Bigr(\sqrt{\frac{\vert\log\varepsilon\vert}{\varepsilon}}\Bigl({u_{N-1}\over N-1}\Bigr)\Bigr)+\\\notag
&\hskip-1cm+\Bigr(\underbrace{J\Bigr(\sqrt{\frac{\vert\log \varepsilon \vert}{\varepsilon}}
(u_N-u_{N-1})\Bigr)\vee J\Bigr(\sqrt{\frac{\vert\log \varepsilon \vert}{\varepsilon}}
(u_N^{\bar{k}}-u_{N-1}^{\bar{k}})\Bigr)\Bigr)}_*
\end{align}
with boundary conditions  $u_0=0$ and $u_N=h(k\tau)$ ($k>\bar{k}$, otherwise we are in the same assumption of (1)). Considering the part $(*)$ in (\ref{probl}) for the last term ensures that the energy spent for the elongation of the $N$-th spring is not reabsorbed.

We denote \begin{align*}
&\bar{w}:=u_N^{\bar{k}}-u_{N-1}^{\bar{k}},\qquad
 w_k:=u_N-u_{N-1}\\
 & z_k={u_{N-1}\over N-1}=\frac{h(k\tau)-w_k}{N-1}
\end{align*}
and rewrite \eqref{probl} as
\begin{align}
\label{probl2}
\min \Bigr\{(N-1)&J\Bigr(\sqrt{\frac{\vert\log\varepsilon\vert}{\varepsilon}}z_k\Bigr)+\Bigr(J\Bigr(\sqrt{\frac{\vert\log \varepsilon \vert}{\varepsilon}}
w_k\Bigr)\vee J\Bigr(\sqrt{\frac{\vert\log \varepsilon \vert}{\varepsilon}}
\bar{w}\Bigr)\Bigr)\Bigr\}\,.
\end{align}
In particular, we note that, when $w_k<\bar{w}$,  minimizing \eqref{probl2}  is equivalent to minimize the contribution of the first $N-1$ springs. Due to the convexity of the function $J$  on the $N-1$ springs, the energy reaches the minimum value  for the minimum value of $z_k$. Now observing that
\begin{align*}
z_k=\frac{h(k\tau)-w_k}{N-1} \geq \frac{h(k\tau)-\bar{w}}{N-1} 
\end{align*}
the minimum is reached for
\begin{align*} 
 z_k=\frac{h(k\tau)-\bar{w}}{N-1} ,
\end{align*}
so that the energy reads
\begin{align*}
F_{\varepsilon}(k)=\begin{cases}
\dfrac{1}{\vert\log\varepsilon\vert}\Bigl((N-1)J\Bigr(\sqrt{\frac{\vert\log\varepsilon\vert}{\varepsilon}}\Bigr(\frac{h(k\tau)-\bar{w}}{N-1}\Bigr)\Bigr)\cr\qquad\qquad\qquad\qquad\qquad +J\Bigr(\sqrt{\frac{\vert\log \varepsilon \vert}{\varepsilon}}
\bar{w}\Bigr)\Bigr)  & \mbox{if} \ h(k\tau)>\bar{w}\\[2ex]
\dfrac{1}{\vert\log\varepsilon\vert}J\Bigr(\sqrt{\frac{\vert\log \varepsilon \vert}{\varepsilon}}
\bar{w}\Bigr) &  \mbox{otherwise.}
\end{cases}
\end{align*}

%\begin{figure}[h]
%\begin{center}
%\includegraphics[scale=0.3]{grafico_cicli4}
%\end{center}
%\caption{Quasi-Static configuration}
%\end{figure}
This description holds as long as $|h(k\tau)|\le h(\bar{k}\tau)$, after which we return to 
case $1$ above.

\begin{figure}[h]
\begin{center}
\includegraphics[scale=0.9]{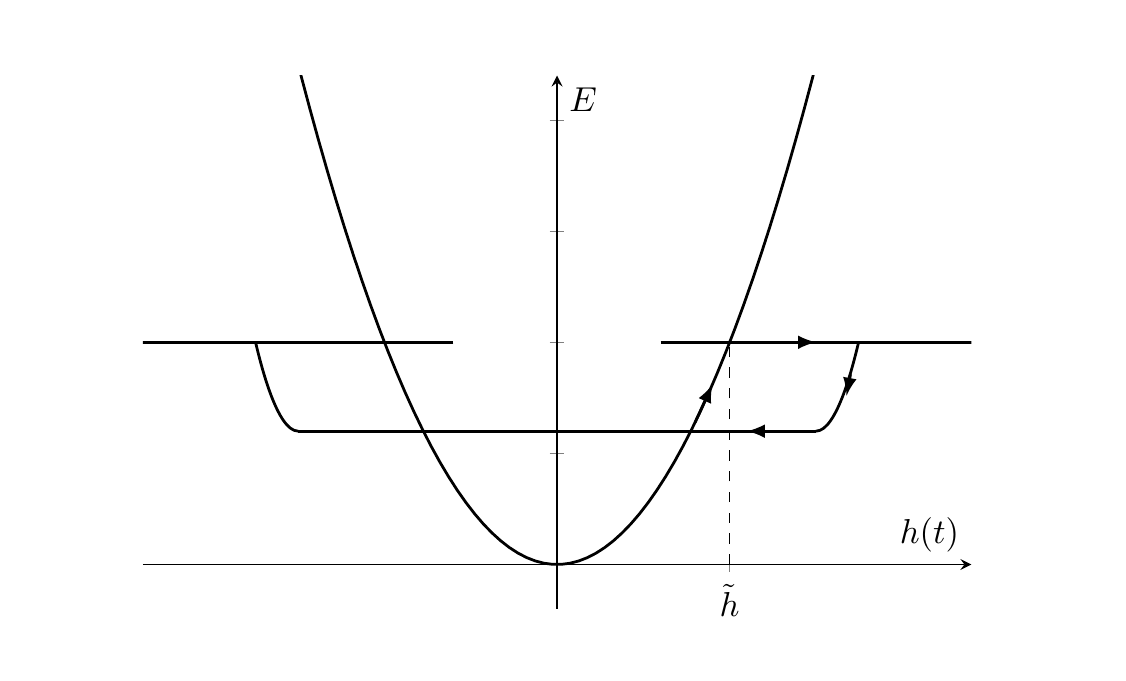}
\end{center}
\caption{Energy of a quasistatic evolution}\label{Fig4}
\end{figure}
\begin{oss}[comparison with the Mumford-Shah quasistatic motion]
We can test the quasistatic behaviour of the Perona-Malik system at fixed $\e$ with 
$h(t)= t_0-|t-t_0|$, and $t_0>1$, so that $t_0>\tilde h_\e$ for $\e$
small enough. If $E_{\tau,\e}^k$ denotes the minimal energy at fixed $\e$ and $\tau$,
by the description above we have
$$
E_{\tau,\e}^k
=\begin{cases} 
\dfrac{1}{\e\vert\log\varepsilon\vert}J\left(h(k\tau)\sqrt{\3}\right) & \hbox{ if } k\tau\le \tilde h_\e
\cr\cr
\dfrac{1}{\vert\log\varepsilon\vert}\Bigl((N-1) J\Bigr(\1 \Bigr(\frac{h(k\tau)-w}{N-1}\Bigr)\Bigr)\cr
\hskip4cm+J\Bigr(\1w\Bigr)\Bigr) & \hbox{ if }  \tilde h_\e<k\tau\le t_0
\cr\cr
\dfrac{1}{\vert\log\varepsilon\vert}\Bigl((N-1)J\Bigr(\sqrt{\frac{\vert\log\varepsilon\vert}{\varepsilon}}\Bigr(\frac{h(k\tau)-\bar{w}}{N-1}\Bigr)\Bigr)\cr
\hskip4cm+J\Bigr(\sqrt{\frac{\vert\log \varepsilon \vert}{\varepsilon}}
\bar{w}\Bigr)\Bigr)  & \mbox{ if } t_0<k\tau\le t_1\cr
\dfrac{1}{\vert\log\varepsilon\vert}J\Bigr(\sqrt{\frac{\vert\log \varepsilon \vert}{\varepsilon}}
\bar{w}\Bigr) & \mbox{ if } t_1<k\tau\le t_2\cr\cr
\dfrac{1}{\vert\log\varepsilon\vert}\Bigl((N-1)J\Bigr(\sqrt{\frac{\vert\log\varepsilon\vert}{\varepsilon}}\Bigr(\frac{h(k\tau)+\bar{w}}{N-1}\Bigr)\Bigr)
\cr
\hskip4cm+J\Bigr(\sqrt{\frac{\vert\log \varepsilon \vert}{\varepsilon}}
\bar{w}\Bigr)\Bigr)  & \mbox{ if } k\tau> t_2\,,
\end{cases}
$$
where $t_1, t_2>t_0$ are satisfy $h(t_1)=\bar w$  and  $h(t_2)=-\bar w$. The values $E_{\tau,\e}^k$
in dependence of $h(t)$ lie in the curves pictured in Fig.~\ref{Fig4}.

As $\tau\to 0$ and $\e\to 0$ the piecewise-constant functions defined by $E_{\tau,\e}(t)=E_{\tau,\e}^{\lfloor t/\tau\rfloor}$ converge to $E$ given by
\begin{align}
\label{Et}
E(t)=\begin{cases} |h(t)|^2 & \hbox{ if } t\le 1\cr
1 & \hbox{ if } t>1.\end{cases}
\end{align}
In fact for $k\tau\le \tilde h_\e$ the energy $E_{\tau,\e}^k$ reads
\begin{eqnarray*}
\dfrac{1}{\e\vert\log\varepsilon\vert}J\left(h(k\tau)\sqrt{\3}\right)&=&\dfrac{1}{\e\vert\log\varepsilon\vert}\log\left(1+h^2(k\tau)\e\vert\log\varepsilon\vert\right)\\&\sim&
\dfrac{1}{\e\vert\log\varepsilon\vert}h^2(k\tau)\e\vert\log\varepsilon\vert\rightarrow h^2(t).
\end{eqnarray*}
For $k\tau> \tilde h_\e$ the contribution of the $N-1$ springs in the convex part vanishes. We make the computation only for one contribution, the others being analogous:
\begin{eqnarray*}
&&\dfrac{1}{\vert\log\varepsilon\vert}(N-1) J\Bigr(\1 \Bigr(\frac{h(k\tau)-w}{N-1}\Bigr)\Bigr)\\
&=&
\dfrac{1-\e}{\e\vert\log\varepsilon\vert}\log\left(1+\dfrac{\vert\log\varepsilon\vert}{\e}\left(\dfrac{\e}{1-\e}\right)^2(h(k\tau)-w)^2\right)\\
&\sim&
\dfrac{1}{1-\e}(h(k\tau)-w)^2\\
&=&\dfrac{1}{1-\e}\left(\dfrac{h^2(k\tau)}{2}-\dfrac{1-\e}{\vert \log \e \vert}-h(k\tau)\sqrt{\dfrac{h^2(k\tau)}{4}-\dfrac{1-\e}{\vert \log \e \vert}}\right)\rightarrow 0.
\end{eqnarray*}
Finally, the spring in the non-convex part gives a constant contribution:
\begin{eqnarray*}
\dfrac{1}{\vert\log\varepsilon\vert}J\Bigr(\sqrt{\frac{\vert\log \varepsilon \vert}{\varepsilon}}w\Bigr)&=&\dfrac{1}{\vert\log\varepsilon\vert}\log\left(1+\dfrac{\vert\log\varepsilon\vert}{\e}w^2\right)\\
&\sim&\dfrac{1}{\vert\log\varepsilon\vert}\left(\log\left(\dfrac{\vert\log\varepsilon\vert}{\e}\right)+\log\left(w^2\right)\right)\\
&=&\dfrac{1}{\vert\log\varepsilon\vert}\left(\log(\vert\log\varepsilon\vert)+|\log(\e)|+\log\left(w^2\right)\right)\rightarrow 1.
\end{eqnarray*}

The energy $E(t)$ corresponds to the energy of the quasistatic motion of the  Mumford-Shah functional with increasing fracture. More general $h$ can be treated analogously
\end{oss}

%In this last case we observe that, if the function $g$ has an undulatory behaviour  (i.e. is increasing with a maximum for $k=\underline{k}$, then it is decreasing, after that we have a new maximum for $k=\hat{k}$ and so on), we have an increasing series of cycles like in the following figure: 

%\begin{figure}[h]
%\begin{center}
%\includegraphics[scale=0.3]{grafico_cicli3}
%\end{center}
%\caption{Example of increasing cycles}
%\end{figure}

%Moreover in the limit for $\varepsilon\to 0$,  we obtain the quasi-static motion of the  Mumford-Shah functional with increasing fracture. 
\section{Dynamic analysis}
We will make use of the method of {\em minimizing movements} along the sequence of functionals $F_\e$ at a scale $\tau=\tau_\e\to 0$ \cite{due}. With varying $\tau$, minimizing movements describe the possible gradient-flow type evolutions along $F_\e$. When $\e\to 0$ fast enough with respect to $\tau$ then we obtain a minimizing movement for the $\Gamma$-limit of $F_\e$; i.e., in our case for the Mumford-Shah functional. In analogy with the result of Braides {\em et al.}~\cite{cinque} we will prove that the restriction that $\e\to0$ fast enough may be removed, so that we may regard the Mumford-Shah functional as a dynamic approximation for $F_\e$. For further properties of minimizing movements for a single energy we refer to \cite{AGS}.

In this section it is useful to express \eqref{pm2} in the following way: we define $f_{\varepsilon}:\R\to \R^+$ by 
\begin{align*}
f_{\varepsilon}(u)= \frac{1}{\3} J\Bigl(\sqrt{\3}u\Bigr)
\end{align*}
and rewrite $F_{\varepsilon}(u)$ as
\begin{align}
\label{F}
F_{\varepsilon}(u)=\sum_{i=1}^{\N_{\varepsilon}}\varepsilon\,f_{\varepsilon}\Bigl(\frac{u_i-u_{i-1}}{\varepsilon}\Bigr).
\end{align}

\begin{oss}
\label{remark1}
If $u:I_{\varepsilon}\to \R$ then, with a little abuse of notation, we will denote with $u$ also the piecewise-constant extension defined by 
$\displaystyle u(x)= u_{\lfloor x/\varepsilon \rfloor}$.
\end{oss}

\subsection{A compactness results}
The paper \cite{cinque} analyzes the dynamic behaviour for functionals similar to (\ref{F}), up to a scaling factor, but with $\widetilde J(z)=\min\{z^2,1\}$ (Blake and Zisserman potential), which has a convex-concave form similar to the Perona-Malik potential. 
A crucial argument in that paper is an observation by Chambolle \cite{Ch} which
allows to identify each function $u$ with a function $v$ such that $F_\e(u)=M_s(v)$ and $v$ is $\e$-close in $L^1$-norm to $u$. In this way the coerciveness properties for the Mumford-Shah functional imply compactness properties for sequences with equibounded energy. 

The Chambolle argument 
simply identifies indices $i$ such that  $(u_{i}-u_{i-1})/\e$ is not in the `convexity region' for the corresponding $f_\e$ with jump points of $v$. This argument is not possible in our case. 
Indeed, let $\varepsilon_n\to 0$ be a vanishing sequence of indices, and let $u_n:I_{\varepsilon_n} \to \R$ be such that

$\bullet$ $\{u_n\}$ is e sequence of equibounded functions;

$\bullet$ $\sup_n F_n(u_n)<M$,  $M>0$ a constant.

%Then we want to prove that for a proper extension of such functions, it is possible to apply the following Theorem (\cite{uno}) 
%\begin{thm}
%\label{AFP}
% Let $\{v_n\}$ be a sequence in $SBV([0,1])$ such that
%\begin{align}
%\sup_{n}\Bigl(\int_a^b|v'_{n}(x)^2|\, dx+\#S(v_{n})+\Vert v_{n}\Vert_{\infty}\Bigr)<+\infty.
%\end{align}
%Then there exists a subsequence $\{v_{n_k}\}$ and a function $u\in SBV([0,1])$ such that
%\begin{align*}
%v_{n_k}\to u, \ \ \ v_{n_k}'\rightharpoonup u' \ \ \ \mbox{in} \ \ L^2(0,1)
%\end{align*}
%Moreover, $D^j v_{n_k}\rightharpoonup D^j u$ weakly-* in the sense of measures.
%\end{thm}
\noindent Define the set 
\begin{align*}
I_n^j=\Bigl\{i\in \Z, \, 0\leq i\leq N_n-1:\frac{\vert (u_n)_{i+1}-(u_n)_i\vert}{\varepsilon_n}>  \frac{1}{\sqrt{\varepsilon_n|\log \varepsilon_n|}}\Bigr\},
\end{align*}
and consider the Chambolle interpolation 
%which is piecewise affine on the point below the convexity threshold and piecewise constant on the point above:
\begin{align}
\label{extension2}
w_n(x):=
\begin{cases}
(u_n)_i \ \ \ \ \ \ \ \ &\mbox{if } i:=\lfloor x/\varepsilon_n\rfloor\in I^j_n \ \ \mbox{ or } \ i=N_n\\
(1-\lambda)(u_n)_i+\lambda (u_n)_{i+1} &\mbox{otherwise} \ \ \ (\lambda:=x/\varepsilon_n-\lfloor x/\varepsilon_n\rfloor)\,.
\end{cases}
\end{align}
Then the set of jump points of $w_n(x)$ may not be bounded as $n\to \infty$, since the only
estimate we may have is
\begin{eqnarray}\nonumber
\label{crescitalogaritmica}
M&\geq& F_n(u_n)\geq \sum_{i\in I^j_n}\varepsilon_n f_n\Bigl(\frac{|(u_n)_{i+1}-(u_n)_{i}|}{\varepsilon_n}\Bigr)\\
&\geq& \sum_{i\in I^j_n}\varepsilon_n f_n\Bigl(\frac{1}{\sqrt{\varepsilon_n|\log \varepsilon_n|}}\Bigr)\geq
\frac{\log(2)}{\vert \log \varepsilon_n \vert}\#\,(I_n^j),
\end{eqnarray}
so that
\begin{align}
\#\,(S(w_n))\leq\#\,(I_n^j)\leq C\vert \log \varepsilon_n \vert.
\end{align}

In order to avoid this obstacle we need to modify the previous sequence. To that end 
it is useful to briefly recall some results due to Morini and Negri \cite{otto}.
\begin{lem} 
\label{lemma1}
Let $p(\varepsilon)>0$ be such that $\lim_{\varepsilon\to 0^+} p(\varepsilon)=0$ and 
\begin{align*}
\lim_{\varepsilon\to 0^+}\Bigl(p(\varepsilon)|\log(\varepsilon)|-\log(|\log \varepsilon|)\Bigr)=+\infty,
\end{align*}
let $c_{\varepsilon}:=\varepsilon^{p(\varepsilon)}$, then it holds that
\begin{align}
&\lim_{\varepsilon\to 0^+} c_{\varepsilon}|\log(\varepsilon)|=0\\
\label{formula5}
&\lim_{\varepsilon\to 0^+} \frac{1}{|\log(\varepsilon)|}J\Bigl(\1c_{\varepsilon}\Bigr)=1.
\end{align}
\end{lem}

We define $b_{\varepsilon}:=(\3)^{1/4}$. 
Denoted $b_n=b_{\e_n}$ and $c_n=c_{\e_n}$, we define  the following sets
\begin{align*}
I_n^1(u_n):&=\Bigl\{i\in\Z, \ 0\leq i\leq N_n-1 : \frac{b_n}{\sqrt{\varepsilon_n |\log \varepsilon_n|}}\leq\frac{\vert (u_n)_{i+1}-(u_n)_i\vert}{\varepsilon_n}\leq \frac{c_n}{\varepsilon_n}\Bigr\}\\
&=\{x_n^1,\dots x_n^{m_n}\}, 
\end{align*}
where $m_n:=\# I_n^1(u_n)$.

The sequence $\{u_n\}$ may be modified into a sequence $\{\tilde{u}_n\}$  such that

1. $I_{n}^1(\tilde{u}_{n})$ is empty;

2. $\Vert\tilde{u}_{n}-u_{n}\Vert_1\to 0$;

3. $F_{n}(\tilde{u}_{n})\leq F_{n}(u_{n})$.

\noindent To that end we define by induction the following sequence
\begin{align*}
&v_n^0\equiv u_n\\
&v_n^{k+1}(t):=
\begin{cases}
v_n^k(t)  &\mbox{if }\  t< x_n^{k+1}+\varepsilon_n\\
v_n^k(t)-[v_n^k( x_n^{k+1}+\varepsilon_n)-v_n^k( x_n^{k+1})] \ \ &\mbox{if }\ t\geq x_n^{k+1}+\varepsilon_n
\end{cases}
\end{align*}
for $k=0,\dots, m_n-1$,
and then we set $\tilde{u}_n:=v_n^{m_n}$. This sequence satisfies all our requests (see \cite{otto}).

\begin{oss}
\label{ossdue}
 It is worth noting that $(\tilde{u}_n)_{i+1}-(\tilde{u}_n)_{i}=(u_n)_{i+1}-(u_n)_{i}$ for indices in $I_n\backslash I_n^1(u_n) $.
\end{oss}
We now consider the following sets
\begin{align*}
&I^2_n:=\Bigl\{i\in\Z, \ 0\leq i\leq N_n-1 :\frac{\vert (\tilde{u}_n)_{i+1}-(\tilde{u}_n)_i\vert}{\varepsilon_n}\leq  \frac{b_n}{\sqrt{\varepsilon_n|\log \varepsilon_n|}}\Bigr\},\\
&I^3_n:=\Bigl\{i\in\Z, \ 0\leq i\leq N_n-1 :\frac{\vert (\tilde{u}_n)_{i+1}-(\tilde{u}_n)_i\vert}{\varepsilon_n}\geq  \frac{c_n}{\varepsilon_n}\Bigr\},
\end{align*}
and the extension  $\tilde{w}_n$  of the function $\tilde{u}_n$ on $[0,1]$ such that $\tilde{w}_n$ is  the affine interpolation of $\tilde{u}_n$ on $I_n^2$ and it is piecewise-constant on $I_n^3$.
\begin{align}
\label{extension}
\tilde{w}_n(x):=
\begin{cases}
(\tilde{u}_{n})_i \ \ \ \ \ \ \ \ &\mbox{if } i:=\lfloor x/\varepsilon_n\rfloor\in I^3_n \ \ \mbox{ or } \ i=N_n\\
(1-\lambda)(\tilde{u}_{n})_i+\lambda (\tilde{u}_{n})_{i+1} &\mbox{otherwise} \ \ \ (\lambda:=x/\varepsilon_n-\lfloor x/\varepsilon_n\rfloor)\,.
\end{cases}
\end{align}
We remark that $I_{\varepsilon_n}=I^2_n\cup I^3_n$, so that $\tilde{w}_n$ is defined for all $x$.

We note that $\tilde{w}_{n}$ still converge to $u_n$ in $L^1$. Moreover it can be proved (see \cite{otto}) that  for a fixed $\delta <<1$ there exists an $\bar{\varepsilon}$ such that for $\varepsilon_n\leq \bar{\varepsilon}$ it holds 
\begin{align}
\label{disuguaglianzaext}
F_n(u_n)\geq (1-\delta)\Bigl(\int_0^1|\tilde{w}_n'|^2\, dx+\mathcal{H}^0(S(\tilde{w}_n))\Bigr),
\end{align}
where $S(\tilde{w}_n)$ is the set of jump points of $\tilde{w}_n$.

Collecting \eqref{extension},\eqref{disuguaglianzaext} and  using the coerciveness properties of the Mumford-Shah energy \cite{uno}, we have the following lemma.
 \begin{lem}
 \label{lemma2}
  Let $\varepsilon_n\to 0$ be a sequence of vanishing indices, $\{u_n\}$ be an equibounded sequence of functions $u_n:I_{\varepsilon_n}\to \R$ such that 
$\sup_n F_n(u_n)\leq M$ for a constant $M>0$. Let $\tilde{w}_n$ be as in \eqref{extension} and suppose that \eqref{disuguaglianzaext} holds. Then, up to a subsequence, there exists a  function  $u\in SBV([0,1])$ such that
 \begin{align*}
 \tilde{w}_n\to u, \ \ \ \ \tilde{w}'_n \rightharpoonup u' \ \ \ \ \mbox{ in } L^2(0,1)\,.
 \end{align*}
 Moreover, $D^j \tilde{w}_n\rightharpoonup D^j u$ weakly-* in the sense of measures.
 \end{lem}
\begin{oss} The previous lemma implies that the sequence $\{u_n\}$ converges to $u$ in $L^2(0,1)$.
Indeed, $\{u_n\}$ are equibounded and $\Vert \tilde{w}_n-u_n\Vert_1\to 0$ for construction, so that there exists a subsequence in $L^{\infty}(0,1)$ which converges to $u$ a.e. The result follows now from Lebesgue's dominated convergence theorem.
\end{oss}

Now, the $L^2$-convergence of $u_n$ implies the $L^2$-convergence of $w_n$ defined in (\ref{extension2}). Indeed, recalling Remark  \ref{remark1} it holds that for every $x\in [0,1]$
\begin{align}
\label{convergenza}
\vert w_n(x)-u_n(x)\vert\leq \sqrt{\frac{\varepsilon_n}{|\log \varepsilon_n|}}.
\end{align}
Moreover  a simple computation as follows shows that $\{w'_n\}$ is equibounded in $ L^2(0,1)$: 
\begin{eqnarray}
\label{weakbound}
M&>&F_n(u_n)\geq \sum_{i\notin I^j_n(w_n)}\varepsilon_n f_n\Bigl(\frac{ (u_n)_{i+1}-(u_n)_i}{\varepsilon_n}\Bigr)\notag\\ &=&\sum_{i\notin I^j_n(w_n)}\frac{1}{\vert\log\varepsilon_n\vert}J\left(\sqrt{\frac{|\log \varepsilon_n|}{\varepsilon_n}} \left((u_n)_{i+1}-(u_n)_i\right)\right)
\notag\\  &\geq&\sum_{i\notin I^j_n(w_n)}\varepsilon_n\Bigl(\frac{ (u_{n})_{i+1}-(u_{n})_i}{\varepsilon_n}\Bigr)^2\geq \int_0^1 |w_{n}'|^2(x)\,dx.
\end{eqnarray}
Hence, there exists a subsequence weakly converging in $L^2(0,1)$. By an integration by parts argument, up to subsequence, we have
\begin{align}
\label{convergenzadebole}
w'_{n}\rightharpoonup u' \ \ \mbox{ in } L^2(0,1).
\end{align}

 The behaviour of points which are above the convexity threshold is of particular interest and it is described in the following lemma.
\begin{lem}
\label{lemma3} Let $\{u_n\}$ be as in Lemma \ref{lemma2} and $w_n$ as in \eqref{extension2}, then, up to subsequence, for every $\bar{x}\in S(u)$ there exists a sequence $\{x^n\}$ converging to $\bar{x}$ such that
\begin{align}
 x^n\in S(w_n) \ \ \ \mbox{and}\ \ \ \lim_{n\to +\infty}|w_n^+(x^n)-w_n^-(x^n)|>\gamma>0.
\end{align}
\end{lem}
\begin{proof}[Proof.] We observe that $S(\tilde{w}_n)\subset S(w_n)$ and that  Lemma \ref{lemma2} holds for $\tilde{w}_n$, so that we can apply the same proof as in Lemma $2.4$ \cite{cinque}. Instead, for points in $S(w_n)\backslash S(\tilde{w}_n)$ it holds that
\begin{align*}
&\lim_{n\to +\infty}\sum_{x\in S(w_n)\backslash S(\tilde{w}_{n}) }|w_{n}^+(x^{n})-w_{n}^-(x^{n})|\leq \lim_{n\to+ \infty} c_n \#\,(S(w_n)\backslash S(\tilde{w}_n))\\
& \leq \lim_{n\to +\infty}  c_n \#\,(I_{n}^j)\leq  \lim_{n\to +\infty} K c_n |\log \varepsilon_n|=0,
\end{align*}
where in the last inequality we use \eqref{crescitalogaritmica}.
\end{proof}

\subsection{Minimizing Movements}
%We briefly recall  some key points of the method of minimizing movements, for more details we invite the reader to consult \cite{due} .\\
%We consider an Hilbert space $X$ and a functional $F$ defined on $X$. To detect the evolution from an initial state $u_0$ we define a time lattice of step $\tau>0$, then we consider a sequence defined as follow: \\
%- Assume $u_{\tau}^0\equiv u_0$.\\
%- Impose that $u_{\tau}^k$ for $k\geq 1$ is a minimizer of the penalized energy
%\begin{align*}
%v\to F(v)+\frac{1}{2\tau}\Vert v-u_{\tau}^{k-1}\Vert^2_{X}.
%\end{align*}
%We interpret $u_{\tau}^k$ as the state of the system at time $t=k\,\tau$. \\
%Now consider the piecewise-constant extension 
%$$u_{\tau}:[0,+\infty)\to X \ \ \mbox{ such that } \ \ u_{\tau}(t)\equiv u_{\tau}^{\lfloor t/\tau\rfloor},$$ then a minimizing movement for $F$ from $u_0$ is a function $u:[0,+\infty]\to X$ which is the pointwise limit of a subsequence $\{u_{\tau_n}\}$ for some $\tau_n\to 0$.\\
%We apply the same method in our situation: we have a sequence of functional $F_{\varepsilon}$ depending on a space parameter $\varepsilon>0$ and we know its $\Gamma$-limit $F$ (\cite{otto}). 
With fixed $\e$ and $\tau=\tau_\e$, from an initial state $u_0^{\varepsilon}: I_{\varepsilon}\to \R$, we define the sequence $u^k:=u^k_{\varepsilon,\tau}$ such that $u^k$ is a minimizer of
\begin{align}
\label{mm}
v\to F_{\varepsilon}(v)+\frac{1}{2\tau}\sum_{i=0}^{N _{\varepsilon}}\varepsilon \vert v_i-u_i^{k-1}\vert^2 \ \ \ \ \forall\, v: I_{\varepsilon} \to \R.
\end{align}
We define the piecewise-constant extension $u_{\varepsilon,\tau}:[0,1]\times[0,+\infty) \to \R$ as
\begin{align}
\label{ext1}
u_{\varepsilon,\tau}(x,t)=(u_{\varepsilon,\tau}^k )_i\ \ \ \mbox{ with } \, k=\lfloor t/\tau\rfloor \ \ \mbox{ and } \ \ i=\lfloor x/\varepsilon\rfloor,
\end{align}
and take the limit (upon extraction of a subsequence) for  both  parameters $\varepsilon\to 0$ and $\tau \to 0$. A limit $u$ is called a {\em minimizing movement} along $F_\e$ at scale $\tau=\tau_\e$.
Observe that in general the limit will depend on the choice of $\varepsilon$ and $\tau$.\\

We now state some properties of minimizing movements \cite{due}.
\begin{prop}
\label{properties}
Let $F_{\varepsilon}$ be as in  \eqref{F} and $u^k$ be defined as above. Then for every $k\in \N$ it holds that
\begin{align*}
&1)\ \ \ \ F_{\varepsilon}(u^k)\leq F_{\varepsilon}(u^{k-1});\\
&2)\ \ \ \ \sum_{i=0}^{N_{\varepsilon}}\varepsilon \vert u^k_i-u^{k-1}_i\vert^2\leq 2\tau [F_{\varepsilon}(u^{k-1})-F_{\varepsilon}(u^k)];\\
&3) \ \ \ \ \Vert u^k \Vert_{\infty}\leq \Vert u^{k-1}\Vert_{\infty}\leq \Vert u^0_{\varepsilon}\Vert_{\infty}.
\end{align*}
\end{prop}
From \eqref{F} and \eqref{mm} we obtain the following optimality conditions.
\begin{prop}
\label{optimality}
 Let $\{u^k\}_k$ be a sequence of minimizer of \eqref{mm}. Then we have
\begin{align*}
-&f_{\varepsilon}'\Bigl(\frac{u^k_1-u^k_0}{\varepsilon}\Bigr)+\frac{\varepsilon}{\tau}(u^k_0-u^{k-1}_0)=0,\\[1.5ex]
 &f_{\varepsilon}'\Bigl(\frac{u^{k}_i-u^k_{i-1}}{\varepsilon}\Bigr)-f_{\varepsilon}'\Bigl(\frac{u^{k}_{i+1}-u^k_i}{\varepsilon}\Bigr)+\frac{\varepsilon}{\tau}(u^{k}_i-u^{k-1}_i)=0,\\[1.5ex]
 &f_{\varepsilon}'\Bigl(\frac{u^k_N-u^k_{N-1}}{\varepsilon}\Bigr)+\frac{\varepsilon}{\tau}(u^k_N-u^{k-1}_N)=0.
\end{align*}
\end{prop}

On the initial states $u^0_{\varepsilon}$ we make the following assumptions.
\begin{align}
\label{ip1}
& \sup\{\vert(u_{\varepsilon}^0)_i\vert \ : 0\leq i \leq N, \ \varepsilon>0\}<\infty.\\
\label{ip2}
& F(u_{\varepsilon}^0)\leq M \ \mbox{ for some } M>0 \ \mbox{ and for every } \varepsilon>0.
\end{align}
Under these hypothesis it is possible to prove the following result.
\begin{thm}
\label{thm1}
Let $\{\varepsilon_n\},\{\tau_n\}\to 0$.
Let $v_{n}=u_{\varepsilon_n,\tau_n}$ be defined as in \eqref{ext1} and consider $\tilde{v}_{n}$ its extension by linear interpolation as in \eqref{extension2}. Then there exist a subsequence of $\{v_n\}$ and a function $u\in C^{1/2}([0,+\infty]; L^2(0,1))$ such that

{\em1)} $v_{n} \to u$ , $\tilde{v}_{n}\to u$ in $L^{\infty}([0,T];L^2(0,1))$ and a.e.~in $(0,1)\times (0,T)$
for every $T\geq 0$;

{\em2)} for all $t\geq 0$ the function $u(\cdot,t)$ is piecewise-$H^1(0,1)$ and $(\tilde{v}_{n})_x(\cdot,t)\rightharpoonup u_x(\cdot,t)$ in $L^2(0,1)$;

{\em3)} for every $\bar{x}\in S(u(\cdot,t))$ there exist a subsequence $\{v_{n_h}\}$ (which can also depend on $t$) and a sequence $(x^h)_h$ converging to $\bar{x}$ such that $x^h\in S(\tilde{v}_{n_h})$.
\end{thm}

\begin{proof}[Proof.]
We will only give a brief sketch of the proof since it follows strictly  the one in \cite{cinque}. 

For fixed $t\geq 0$ the equiboundedness of initial data guarantees that also $F_{\varepsilon}(v_n(\cdot,t))$ is bounded, so that we are in the hypotheses of the previous section and we can apply the same construction to the sequence $\{v_n(\cdot,t)\}_n$. Having in mind \eqref{convergenza} and \eqref{convergenzadebole}, this shows that up to subsequences, $\tilde{v}_n(\cdot,t)$ is converging in $L^2(0,1)$ to a function $u(\cdot,t)$ piecewise-$H^1(0,1)$ and also $(\tilde{v}_n)_x(\cdot,t)$ is weakly converging in $L^2(0,1)$ to $u_x(\cdot,t)$.
Now, a well-known result for minimizing movements (for example \cite{due}, \cite{cinque}) proves that
\begin{align}
\Vert v_n(\cdot,t)-v_n(\cdot,s)\Vert_2\leq C\sqrt{t-s-\tau_n}
\end{align}
that in the limit becomes
\begin{align}
\Vert u(\cdot,t)-u(\cdot,s)\Vert_2\leq C\sqrt{t-s}
\end{align}
with $C$ independent form both $t$ and $s$. So that the limit function $u$ belongs to $C^{1/2}([0,+\infty]; L^2(0,1))$.

We prove the convergence in $L^{\infty}([0,T];L^2(0,1))$: for $T>0$ fixed, consider $M\in \N$ and $t_j=jT/M$ for $j=0,\dots,M$. Then for every $t\in [0,T]$ there exists a $j=0,\dots, M$ such that $t_{j-1}<t<t_j$, so we have that
\begin{align}
\Vert v_n(\cdot,t)-u(\cdot,t)\Vert_2&\leq \Vert v_n(\cdot,t)-v_n(\cdot,t_{j-1})\Vert_2+\Vert v_n(\cdot,t_{j-1})-u(\cdot,t_{j-1})\Vert_2\notag\\
&+\Vert u(\cdot,t_{j-1})-u(\cdot,t)\Vert_2\notag\\
&\leq 2C\sqrt{t-t_j+\tau_n}+\Vert v_n(\cdot,t_{j-1})-u(\cdot,t_{j-1})\Vert_2.
\end{align}
Since $v_n(\cdot,t)$ is a converging sequence to $u(\cdot,t)$, for $\overline{n}>>1$ it is possible to find an $\eta<<1$ such that
$\Vert v_n(\cdot,t_{j-1})-u(\cdot,t_{j-1})\Vert_2\leq \eta$ for all $n\geq \overline{n}$.
Finally, we have that
\begin{align*}
\sup_{t\in[0,T]}\Vert v_n(\cdot,t)-u(\cdot,t)\Vert_2\leq  2C\sqrt{(T/M)+\tau_n}+ \eta
\end{align*}
for all $n\geq \overline{n}$, which means
\begin{align*}
\limsup_{n\to +\infty }\sup_{t\in[0,T]}\Vert v_n(\cdot,t)-u(\cdot,t)\Vert_2\leq  2C\sqrt{(T/M)+\tau_n}+ \eta
\end{align*}
and the claims now follows from the arbitrariness of $M$ and $\eta$.

We conclude observing that (3) follows from Lemma \ref{lemma3}.
\end{proof}

\subsection{Computation of the limit equation}
Consider now two sequences of indices $\{\varepsilon_n\}\to 0$, $\{\tau_n\}\to 0$ (to simplify the notation from now on we will write $\varepsilon$ instead of $\varepsilon_n$ and similarly $\tau$ instead of $\tau_n$). We define the function
\begin{align}
\label{fin}
\phi_n(x,t):= f_{\varepsilon}'\Bigl(\frac{(u_{\varepsilon,\tau}^k)_{i+1}-(u_{\varepsilon,\tau}^k)_i}{\varepsilon}\Bigr) \ \ \mbox{if} \ i=\lfloor x/\varepsilon \rfloor \ \ \mbox{and}\  k=\lfloor t/\tau \rfloor.
\end{align}
\begin{prop}
\label{3.3}
If $\phi_n$ is defined in \eqref{fin}, then for every $t\geq 0 $  we have
\begin{align*}
\phi_n(\cdot,t)\rightharpoonup 2u_x(\cdot,t) \ \ \ \mbox{in }\ \  L^2(0,1).
\end{align*}
Moreover, for every $T >0$ the sequence $\{\phi_n(\cdot,t)\}$ is uniformly bounded in $L^2(0,1)$ with respect to $t\in[0,T]$ and $u_x\in L^2((0,1)\times (0,T))$.
\end{prop}
\begin{proof}[Proof.] Let $t\geq 0$ fixed, $v_n:=v_{\varepsilon_n}$ and $\tilde{v}_n:=\tilde{v}_{\varepsilon_n}$ be as defined in Theorem \ref{thm1}. Consider the function
\begin{align*}
\chi_n(x)=\begin{cases}
1 \ \ &\mbox{if } \ \ x\in \bigcup_{i\in I_{\varepsilon}^j(v_n(\cdot,t))}\varepsilon[i,i+1)\\
0 & \mbox{otherwise}
\end{cases}
\end{align*}
and the decomposition $\phi_n(\cdot,t)=\chi_n \phi_n(\cdot,t)+(1-\chi_n)\phi_n(\cdot,t)$. From \eqref{crescitalogaritmica} and Proposition \ref{properties} we get that
\begin{align*}
\int_0^1 \vert\chi_n\phi_n(x,t)\vert^2 dx=\sum_{i\in I^j_{\varepsilon}} \varepsilon \vert\phi_n(i\varepsilon,t)\vert^2
\leq\varepsilon \#\,( I^j_{\varepsilon}) f_{\varepsilon}'\Bigl(\frac{1}{\sqrt{\3}}\Bigr)^2\leq M, 
 \end{align*}
 which means that the sequence is $L^2$-bounded. Moreover,
 \begin{align}
 \int_0^1 \vert\chi_n\phi_n(x,t)\vert dx
\leq\varepsilon \#\,( I^j_{\varepsilon}) f_{\varepsilon}'\Bigl(\frac{1}{\sqrt{\3}}\Bigr)\leq M \sqrt{\3}\to 0.
 \end{align}
 This prove that
 \begin{align*}
 \chi_n\phi_n(x,t)\rightharpoonup 0 \ \ \mbox{in } \ L^2(0,1).
 \end{align*}
 
 We now obtain a similar result for $(1-\chi_n)\phi_n(\cdot,t)$. First at all we observe that
 \begin{align}
 \label{vn}
 (\tilde{v}_n)_x(x,t)= \begin{cases}
 \dfrac{(u_{\varepsilon,\tau}^k)_{i+1}-(u_{\varepsilon,\tau}^k)_i}{\varepsilon} & x\in [i\varepsilon,(i+1)\varepsilon) \mbox{ and } i\notin I^j_{\varepsilon}\\[1.5ex]
 0 & \mbox{ otherwise }
 \end{cases}
 \end{align}
 This means that $(1-\chi_n)\phi_n(x,t)=f_{\varepsilon}'((\tilde{v}_n)_x(x,t))$. Using a Taylor expansion of $f_{\varepsilon}'$ in a neighbourhood of the origin we get
 \begin{align*}
 f_{\varepsilon}'(\tilde{v}_n)_x(x,t))=f_{\varepsilon}'(0)+f_{\varepsilon}''(0)(\tilde{v}_n)_x(x,t)+\frac{1}{2} f_{\varepsilon}'''(\xi_n)((\tilde{v}_n)_x(x,t))^2 
 \end{align*}
 for some $\xi_n\in [0,(\tilde{v}_n)_x(x,t)]$, so that
 \begin{align}
 \label{form}
 f_{\varepsilon}'((\tilde{v}_n)_x(x,t))= 2(\tilde{v}_n)_x(x,t)+\frac{1}{2}\sqrt{\3}J'''(\sqrt{\3}\xi_n)((\tilde{v}_n)_x(x,t))^2 .
 \end{align}
Moreover, recalling \eqref{vn}, we have
\begin{align*}
-\frac{1}{\sqrt{\3}}\leq (\tilde{v}_n)_x(x,t) \leq \frac{1}{\sqrt{\3}},
\end{align*}
so that the sequence $\sqrt{\3}(\tilde{v}_n)_x(x,t)$ is bounded, as is $J'''(\sqrt{\3}\xi_n)$. From this, it follows that there exists a constant $C>0$ such that
$$\vert f_{\varepsilon}'((\tilde{v}_n)_x(x,t))\vert\leq C \vert(\tilde{v}_n)_x(x,t)\vert.$$

Now, fix $T>0$ and $t\in [0,T]$.
We proved in \eqref{weakbound} that $(\tilde{v}_n)_x(x,t)$ is bounded in $L^2(0,1)$ and from the above inequality this implies that also $f_{\varepsilon}'((\tilde{v}_n)_x(x,t))$ is bounded in the same space. So there exists at least a subsequence weakly converging in $L^2(0,1)$. We will show that the entire sequence is weakly convergent, i.e.
\begin{align*}
f_{\varepsilon}'((\tilde{v}_n)_x(x,t))\rightharpoonup 2 u_x(x,t) \ \ \ \mbox{in } \ L^2(0,1).
\end{align*}
Recalling now Theorem \ref{thm1}, we observe that in \eqref{form} the right-hand side is weakly converging to $ 2 u_x(x,t)$ in $L^1(0,1)$. Indeed, notice that $J'''(0)=0$ and $(\tilde{v}_n)_x(x,t)$ is equibounded in $L^2(0,1)$.
Hence, we can conclude that 
\begin{align*}
\phi_n(x,t) \rightharpoonup 2 u_x(x,t)\ \ \ \mbox{in } \ L^2(0,1).
\end{align*}
Finally we have that 

$\bullet$ $\chi_n \phi_n$ is uniformly bounded in $L^2(0,1)$;

$\bullet$ $(1-\chi_n)\phi_n$ is itself uniformly bounded because it is $f_{\varepsilon}'((\tilde{v}_n)_x(x,t))$.

\noindent This means that also $\phi_n(x,t)$ is uniformly bounded  in $L^2(0,1)$.
\end{proof}

We can improve the result above. In particular, we may deduce which boundary conditions are satisfied by the weak-limit of $\phi_n(x,t)$. To that end, in the following, we extend definition \eqref{ext1} by setting
\begin{align}
\label{est}
(u_{\varepsilon,\tau}^k)_i=
\begin{cases}
(u_{\varepsilon,\tau}^k)_0 \ \ \ &\mbox{if } \ i\in\Z, \ i<0\\
(u_{\varepsilon,\tau}^k)_N \ \ \ &\mbox{if } \ i\in\Z, \ i>N.
\end{cases}
\end{align}

\begin{thm} Consider a sequence of function $v_n$ as defined in Theorem \ref{thm1}. Let $u$ be its strong limit in $L^2(0,1)$, then

{\em 1)} $u_x(\cdot,t)\in H^1(0,1)$ for a.e.~$t\geq 0$ and $(u_x)_x \in L^2((0,1)\times(0,T))$ for every $T>0$;

{\em 2)} for a.e.~$t\geq 0$ the function $u$ satisfies the boundary conditions
\begin{align*}
u_x(0,t)=u_x(1,t)=0 \ \mbox{and} \ u_x(\cdot,t)=0 \ \mbox{on } \ S(u(\cdot,t)).
\end{align*}
\end{thm} 
\begin{proof}[Proof.]
Let ${\widetilde\phi}_n$ be the linear interpolation  of the function $\phi_n$ defined in \eqref{fin}. Our first claim is that ${\widetilde\phi}_n\rightharpoonup J''(0)u_x(x,t)$ in $H^1(0,1)$.

We recall that, from Proposition \ref{properties}, it holds
\begin{align*}
\sum_{i=0}^N \varepsilon \vert(u_{\varepsilon,\tau}^k)_i-(u_{\varepsilon,\tau}^{k-1})_i\vert^2\leq 2\tau [F_{\varepsilon}(u_{\varepsilon,\tau}^{k-1})-F_{\varepsilon}(u_{\varepsilon,\tau}^k)],
\end{align*} 
so that,  fixing $T>0$ and denoting $N_{\tau}=\lfloor T/\tau\rfloor$, we have
\begin{align*}
\sum_{k=1}^{N_{\tau}}\sum_{i=0}^{N_\varepsilon} \tau \varepsilon \vert(u_{\varepsilon,\tau}^k)_i-(u_{\varepsilon,\tau}^{k-1})_i\vert^2\leq 2\tau^2 F_{\varepsilon}(u_{\varepsilon}^0)\leq 2\tau^2 M. 
\end{align*}
Using the optimality conditions in Proposition \ref{optimality} and the extension \eqref{est}, we get\begin{align*}
\sum_{k=1}^{N_{\tau}}\tau \sum_{i\in \Z}\varepsilon\tau^2 \Bigl[\frac{1}{\varepsilon}\Bigl(f_{\varepsilon}'\bigl(\frac{(u_{\varepsilon,\tau}^k)_{i+1}-(u_{\varepsilon,\tau}^k)_i}{\varepsilon}\bigr)-f_{\varepsilon}'\bigl(\frac{(u_{\varepsilon,\tau}^k)_i-(u_{\varepsilon,\tau}^k)_{i-1}}{\varepsilon}\bigr)\Bigr)\Bigr]^2\leq 2\tau^2 M\,.
\end{align*}
Taking the extension by linear interpolation ${\widetilde\phi_n}$ on $I_{\varepsilon}$ into account, we rewrite the previous estimate as
\begin{align*}
\sum_{k=1}^{N_{\tau}}\tau \int_{\R}[({\widetilde\phi_n})_x(x,k\tau)]^2\,dx\leq 2M,
\end{align*}
so that for $\delta>0$ and $\tau<\delta$ we obtain
\begin{align*}
\int_{\delta}^Tdt\int_{\R}[({\widetilde\phi_n})_x(x,k\tau)]^2\,dx\leq 2M
\end{align*}
and 
\begin{align*}
\liminf_{n\to +\infty}\int_{\delta}^Tdt\int_{\R}[({\widetilde\phi_n})_x(x,k\tau)]^2\,dx\leq 2M.
\end{align*}
By Fatou's Lemma
\begin{align*}
\int_{\delta}^T\Bigl(\liminf_{n\to +\infty}\int_{\R}[({\widetilde\phi_n})_x(x,k\tau)]^2\,dx\Bigr)dt\leq 2M;
\end{align*}
in particular this means that
\begin{align}
\liminf_{n\to +\infty}\int_{\R}[({\widetilde\phi_n})_x(x,k\tau)]^2\,dx< \infty \ \ \ \mbox{ for a.e. } t\in [\delta,T].
\end{align}
Let $t$ be such that the previous inequality holds and consider a subsequence $({\widetilde\phi_{n_k}})_x(x,k\tau)$ such that
\begin{align*}
\liminf_{k\to +\infty}\int_{\R}[({\widetilde\phi_{n_k}})_x(x,k\tau)]^2\,dx=\liminf_{n\to +\infty}\int_{\R}[({\widetilde\phi_n})_x(x,k\tau)]^2\,dx.
\end{align*}
Then there exists $C$ independent of $k$ such that 
\begin{align}
\label{30}
\int_{\R}[({\widetilde\phi_{n_k}})_x(x,k\tau)]^2\,dx\leq C.
\end{align}
We recall that, by Proposition \ref{3.3}, in $L^2(0,1)$ we have
\begin{align*}
\phi_n(\cdot,t)\rightharpoonup \phi(\cdot,t)=
\begin{cases} J''(0)u_x(\cdot,t) \ \ \ \ &\mbox{in } (0,1)\\
0 &\mbox{otherwise}.
\end{cases}
\end{align*} 
The same result also holds for $\widetilde\phi_{n_k}(\cdot,t)$ observing that from \eqref{30}  we get
\begin{align*}
\sum_{i\in\Z:\varepsilon i\in I_{\varepsilon}}\varepsilon|\phi_{n_k}((i+1)\varepsilon,t)-\phi_{n_k}(i\varepsilon,t)|^2\leq \varepsilon^2C.
\end{align*}
Moreover, the $L^2$-weak-convergence of ${\widetilde\phi_{n_k}}$, the boundedness proved in \eqref{30} and an integration by parts argument show that, for any open interval $I\subset [0,1]$,  $\phi\in H^1(I)$ and
\begin{align}
({\widetilde\phi_{n_k}})_x(\cdot,t)\rightharpoonup \phi_x(\cdot,t) \ \ \ \mbox{in } L^2(0,1).
\end{align}
The above observation and convergence result prove that $u_x\in H^1(0,1)$ with $u_x(0,t)=u_x(1,t)=0$ for almost every $t\geq 0$. 

Now, we want to show that $u_{xx}\in L^2((0,1)\times(0,T))$: since
\begin{align*}
\int_0^1 [\phi_x(x,t)]^2\,dx\leq \liminf_{k\to +\infty}\int_{\R}[({\widetilde\phi_{n_k}})_x(x,t)]^2\,dx=\liminf_{n\to +\infty}\int_{\R}[({\widetilde\phi_{n}})_x(x,t)]^2\,dx
\end{align*}
then for every $\delta>0$
\begin{align*}
\int_{\delta}^T dt\int_0^1 [\phi_x(x,t)]^2\,dx\leq\int_{\delta}^T \Bigl(\liminf_{n\to +\infty}\int_{\R}[({\widetilde\phi_{n}})_x(x,t)]^2\,dx\Bigr) dt\leq 2M.
\end{align*}

To conclude we want to understand which values  $u_x(\cdot,t)$ attains on $S(u)$.
Let $t$ be such that \eqref{30} still holds. Observing that $H^1(0,1)\subset\subset C([0,1])$, we have
\begin{align*}
{\widetilde\phi_{n_k}}(x,t) \rightarrow 2u_x(\cdot,t)\ \ \ \ \ \mbox{ in } C([0,1]).
\end{align*}
Now if $\bar{x}\in S(u)$, thanks to Lemma \ref{lemma3} we know that there exists a sequence $\{x^n\}$ converging to $x$ and such that for every $n$ 
\begin{align*}
x^n\in S(\tilde{v}_n(\cdot,t)) \ \ \ \mbox{ and } |\tilde{v}_n^+(x^n,t)-\tilde{v}_n^-(x^n,t)|\geq \gamma >0.
\end{align*}
We observe that, if  $x^n=i_n\varepsilon\in S(\tilde{v}_n(\cdot,t))$, then  $(i_n-1)\in I^j_{\varepsilon}(v_n(\cdot,t))$, so that we can write
\begin{align*}
|{\widetilde\phi}_n((i_n-1)\varepsilon,t)|&=|\phi_n((i_n-1)\varepsilon,t)|=f_{\varepsilon}'\Bigl(\frac{|\tilde{v}^+_n(x^n,t)-\tilde{v}^-_n(x^n,t)|}{\varepsilon}\Bigr)\\
&\leq f_{\varepsilon}'\Bigl(\frac{\gamma}{\varepsilon}\Bigr)=\frac{1}{\sqrt{\3}}J'\Bigl(\1\gamma\Bigr).
\end{align*}
Recalling that $J'(z)=\frac{2z}{1+z^2}$,  a simple calculation shows that
\begin{align*}
|{\widetilde\phi}_n((i_n-1)\varepsilon,t)|\leq \frac{2\gamma}{\varepsilon+\gamma^2\vert\log\varepsilon\vert}\rightarrow 0 \ \mbox{if} \ \varepsilon_n\to 0.
\end{align*}
Now the uniform convergence of ${\widetilde\phi}_n(\cdot,t)$ to $J''(0)u_x(\cdot,t)$  implies that $u_x(\bar{x},t)=0$.
\end{proof}

We conclude this section collecting all the previous results to obtain the limit equation satisfied by a minimizing movement of functional \eqref{pm} (see the analog result in \cite{cinque}).
\begin{thm}
\label{equazionefinale}
Let $\{u^0_{\varepsilon}\}_{\varepsilon}$ be a sequence of functions which satisfies \eqref{ip1} and \eqref{ip2}  and let $v_n=u_{\varepsilon_n,\tau_n}$ be a sequence converging to $u$ as in Theorem $\ref{thm1}$. Then we have
\begin{align}
u_t=2\,u_{xx}
\end{align}
in the distributional sense in $(0,1)\times (0,+\infty)$. Moreover
\begin{align*}
&u(.,0)=u^0 \ \ \mbox{a.e. in } (0,1)\\
&u_x(.,t)=0 \ \ \mbox{on } S(u(.,t))\cup (0,1) \ \mbox{ for a.e.} \ t\geq 0,
\end{align*}
where $u^0$ is the a.e.-limit of the sequence $\{u^0_{\varepsilon}\}_{\varepsilon}$ .
\end{thm}
\subsection{Evolution of the singular set}
Let $u^0_{\varepsilon}$ be an initial datum satisfying \eqref{ip1} and \eqref{ip2}, and consider the sequence $\{u_{\varepsilon,\tau}^k\}$ of minimizers of functional \eqref{mm}. In this section we want to understand the behaviour of the set of singular points  $I^j_{\varepsilon}(u_{\varepsilon,\tau}^k)$ with respect to $k$.

We simplify the notation introducing
$$
u^k_i:=(u_{\varepsilon,\tau}^k)_i,\qquad
v^k_i:=\frac{u^k_{i+1}-u^k_i}{\varepsilon}.
$$
Using the optimality condition for $0<i<N-1$, we observe that 
%(\cite{cinque}):
\begin{align*}
v_i^{k+1}-v_i^k&=\frac{1}{\varepsilon}\Bigl(u^{k+1}_{i+1}-u^{k+1}_i-u^k_{i+1}+u^k_i\Bigr)\\
&=\frac{1}{\varepsilon}\Bigl(u^{k+1}_{i+1}-u^k_{i+1}\Bigr)-\frac{1}{\varepsilon}\Bigl( u^{k+1}_i-u^k_i\Bigr)\\
&=\frac{\tau}{\varepsilon^2}\Bigl[f'_{\varepsilon}\Bigl(\frac{u^{k+1}_{i+2}-u^{k+1}_{i+1}}{\varepsilon}\Bigr)+ f'_{\varepsilon}\Bigl(\frac{u^{k+1}_{i}-u^{k+1}_{i-1}}{\varepsilon}\Bigr)-2f'_{\varepsilon}\Bigl(\frac{u^{k+1}_{i+1}-u^{k+1}_{i}}{\varepsilon}\Bigr)\Bigr].
\end{align*}
Rewriting these terms we obtain
\begin{align}
\label{final1}
(v_i^{k+1}-v_i^k)+2\frac{\tau}{\varepsilon^2}f'_{\varepsilon}\Bigl(\frac{u^{k+1}_{i+1}-u^{k+1}_{i}}{\varepsilon}\Bigr)\leq 2\frac{\tau}{\varepsilon^2}\max f'_{\varepsilon}.
\end{align}
Observe that this estimate still holds for $i=0$ or $i=N-1$. In fact, in that case we have
\begin{align*}
(v_i^{k+1}-v_i^k)+2\frac{\tau}{\varepsilon^2}f'_{\varepsilon}\Bigl(\frac{u^{k+1}_{i+1}-u^{k+1}_{i}}{\varepsilon}\Bigr)\leq \frac{\tau}{\varepsilon^2}\max f'_{\varepsilon}\,.
\end{align*}
Consider now the function 
\begin{align*}
h(z)=2\frac{\tau}{\varepsilon^2}f'_{\varepsilon}(z)
\end{align*}
so that \eqref{final1} reads
\begin{align}
(v_i^{k+1}-v_i^k)+h(v_i^{k+1})\leq \max h.
\end{align}
The key point is the following lemma (see \cite{cinque}).
\begin{lem}Let $h:\R\to\R$ be a Lipschitz function with Lipschitz constant $L<1$. Let $\{a_k\}$ be a sequence of real numbers and $C\in \R$ such that
\begin{align*}
a_{k+1}-a_k+h(a_{k+1})\leq h(C) \ \ \ \hbox{for all }k\geq 0.
\end{align*}
Then, if $a_0\leq C$, it holds $a_k\leq C$ for all $k\geq 0$.
\end{lem}
First we prove that $h$ is a Lipschitz function with $L<1$:
\begin{align*}
h(z_1)-h(z_2)= 2\frac{\tau}{\varepsilon^2}\Bigl(f'_{\varepsilon}(z_1)-f'_{\varepsilon}(z_2)\Bigr)\leq 2\frac{\tau}{\varepsilon^2}f''_{\varepsilon}(\xi)(z_1-z_2) \ \ \ \ \xi \in (z_1,z_2)
\end{align*}
so  we have to require that
\begin{align}
\label{condizionefinale}
2\frac{\tau}{\varepsilon^2}\max f''_{\varepsilon}=2\frac{\tau}{\varepsilon^2}\max J''=4\frac{\tau}{\varepsilon^2}<1
\end{align} 
When this condition is satisfied, choosing $C=1/\sqrt{\3}$, we have the following result.
\begin{prop} If \eqref{condizionefinale} holds, then for every $k\geq 0$
\begin{align*}
I^j_{\varepsilon}(u^{k+1}_{\varepsilon,\tau})\subseteq I^j_{\varepsilon}(u^{k}_{\varepsilon,\tau})
\end{align*}
\end{prop}
We conclude with the following observations.

Let $u(\cdot, t)$ be a minimizing movement for the scaled Perona-Malik functional, then for every $t\geq 0$ each jump point of $u$ is obtained as limit of a sequence of jump points of the extension $\tilde{v}_{\varepsilon}$ as defined in Theorem \ref{thm1} (notice in particular statement $(3)$). Hence, if we define 
\begin{align}
S:=\Bigl\{x_j\in \R : x_j=\lim_{\varepsilon\to 0}x^{\varepsilon}_j, \ x^{\varepsilon}_j\in S(\tilde{v}_{\varepsilon}) \Bigr\}
\end{align}
we have $S(u(\cdot,t))\subseteq S$ for every $t\geq 0$.

Recalling Theorem \ref{equazionefinale}, we proved that a minimizing movement for functional \eqref{pm2}  satisfies the heat equation with Neumann boundary conditions on $(0,1)\backslash S(u^0)$. This result is the same for the minimizing movement of Mumford-Shah functional \cite{due}, so that the following result holds.

\begin{cor} Let $u_{\varepsilon}^0$ be an initial datum that satisfies \eqref{ip1} and \eqref{ip2}; let $u_{\varepsilon}$ be a minimizing movement for the scaled Perona-Malik functional $F_{\varepsilon}$ as defined in \eqref{pm2}. Then, if $\varepsilon\to 0$ and \eqref{condizionefinale} holds, $u_{\varepsilon}$ converges in $L^{\infty}((0,T);L^2(0,1))$ to a minimizing movement of the Mumford-Shah functional.
\end{cor}

\subsection{Long-time behaviour}
In this section we finally remark that, while we concluded that the Perona-Malik energies are approximated by the Mumford-Shah functional as gradient-flow dynamics are concerned, this does not hold for long-time dynamics.

Long-time dynamics can be defined by introducing a time-scaling parameter $\lambda >0$,
and applying a recursive minimizing scheme to the scaled energies. Fixed
 an initial datum $x_0$ we define recursively  $x_k$ as a minimizer for the minimum problem
\begin{align}
\min \left\{\frac{1}{\lambda}F_{\varepsilon}(x)+\frac{1}{2\tau}\Vert x-x_{k-1}\Vert^2\right\}.
\end{align}
Equivalently the same minimum problem can be written as
\begin{align}
\min \left\{F_{\varepsilon}(x)+\frac{\lambda}{2\tau}\Vert x-x_{k-1}\Vert^2\right\}.
\end{align}
so that $x_k$ can be seen as produced by a minimizing movements scheme with time step $\eta=\tau/\lambda$. Now, if $u^{\eta}$ is a discretization over the lattice of time-step $\eta$, we have 
\begin{align*}
u^{\tau}(t):=x_{\lfloor t/\tau \rfloor}=x_{\lfloor t/\lambda\eta \rfloor}=u^{\eta}\Bigl(\frac{t}{\lambda}\Bigr).
\end{align*}
This shows that the introduction of the constant parameter  $\lambda$ is equivalent to a scaling in time.

\smallskip
In order to show that for some time scaling $\lambda$ the sequence $F_\e$ is not equivalent to $M_s$
we consider an initial datum $u_0$ which is a local minimum for $M_s$, so that the corresponding motion is trivial at all scales: $u(t)=u_0$ for all $t$. We then exhibit some $\lambda$ such that the recursive minimization scheme above gives a non-trivial limit evolution.

We consider additional constraints on the domain of $F_\e$ by limiting the test function to local minimizers of $M_s$ with prescribed boundary conditions. More precisely,

$\bullet$ the initial datum $u_0$ is a  a piecewise-constant function with $S(u_0)=\{x_0,x_1\}$ and $0< x_0< x_1 <1$;

$\bullet$ competing functions are non-negative piecewise-constant functions with $S(u_k)\subseteq S(u_0)$;

$\bullet$ boundary conditions read $u(0^-)=0$ and $u(1^+)=1$.

We may describe the minimizers $u_k$ by a direct computation using the minimality conditions:
if $z_k$ is the constant value of $u_k$ on the interval $(x_0,x_1)$, then $z_k$ solves  the equation
\begin{align*}
(x_1-x_0)\frac{z_k-z_{k-1}}{\tau}=-\frac{2}{\lambda}\Bigl(\frac{z_k}{\varepsilon+|\log\varepsilon|z_k^2}+\frac{z_k-1}{\varepsilon+|\log\varepsilon|(z_k-1)^2}\Bigr).
\end{align*}
In order to obtain a non-trivial limit as $\varepsilon,\tau\to 0$,we may choose the scaling 
\begin{equation}\label{ltlambda}
\lambda=\frac{1}{|\log\varepsilon|}.
\end{equation}
 With such a time-scaling, in the limit we get an equation for $z(t)$ of the form
\begin{equation}\label{lteq}
z'=-\frac{2}{(x_1-x_0)}\cdot \frac{1-2z}{z(1-z)}.
\end{equation}
Hence, if the initial datum has the value $z_0\neq1/2$ in the interval $(x_0,x_1)$, the motion is not trivial. We refer to \cite{due} for further examples.

\begin{oss}[equivalent energies for long-time motion]
Note that the time-scaled minimizing-movement scheme along the functionals
$G_\e$ in Section \ref{gammaec} for $\lambda$ as in (\ref{ltlambda}) gives the same limit equation (\ref{lteq}) for the computation above, provided that also $g'(w)\sim {2\over w}$ as $w\to+\infty$. This suggests that $G_\e$ may be considered as a finer approximation carrying on the equivalence to long-time behaviours.
\end{oss}


\begin{thebibliography}{25} 

\bibitem{AB} Ambrosio L. and Braides A. \textit{Energies in SBV and variational models in fracture mechanics}.
In {\em Homogenization and Applications to Material Sciences}, (D. Cioranescu, A. Damlamian, P. Donato eds.), GAKUTO, Gakk$\bar o$tosho, Tokio, Japan, 1997, p. 1--22

\bibitem{uno} Ambrosio L., Fusco N., and Pallara, D.  \textit{Function of Bounded Variations and Free Discontinuity Problems}. Oxford University Press, Oxford, 2000.

\bibitem{AGS} Ambrosio L., Gigli N., and Savar\'e G.   
{\em Gradient Flows in Metric Spaces and in the Space of Probability Measures}. 
Lectures in Mathematics ETH, Z\"urich. Birkhh\"auser, Basel, 2008

\bibitem{BF} Bellettini G., Fusco, G. \textit{The $\Gamma$-limit and the related gradient flow for singular perturbation functionals of Perona-Malik type}.  Trans. American Math. Soc. 360 (2008), 4929--4987

\bibitem{BZ} Blake A. and Zisserman A. {\em Visual Reconstruction}. MIT Press, Cambridge, 1987

\bibitem{BFM} B. Bourdin B., Francfort, G.A., and Marigo, J.-J. 
{\em The Variational Approach to Fracture}.   J. Elasticity 91 (2008),  5--148

\bibitem{tre} Braides, A.  \textit{$\Gamma$-convergence for Beginners}. Oxford University Press (2002)

\bibitem{due} Braides, A.  \textit{Local Minimization, Variational Evolution and $\Gamma$-convergence}. Lecture Notes in Math. 2094, Springer-Verlag, Berlin, 2014

\bibitem{BCGS}Braides A., Colombo M., Gobbino M., and Solci M.
\textit{Minimizing movements along a sequence of functionals and curves of maximal slope}.
 C. R. Acad. Sci. Paris, Ser. I 354 (2016),  685--689

\bibitem {quattro} Braides, A., Dal Maso, G., and  Garroni, A. 
\textit{Variational formulation of softening phenomena in Fracture Mechanics: the one-dimensional case}.
 Arch. Rational Mech. Anal. 146 , 23--58

\bibitem{cinque} Braides, A., Defranceschi, A., and  Vitali, E. 
\textit{Variational Evolution of One-Dimensional Lennard-Jones Systems}.
 Netw. Heterog. Media. 9 (2014), 217--238

\bibitem{BLO}Braides, A., Lew A.J., and Ortiz M.
\textit{Effective cohesive behavior of layers of interatomic planes}.
 Arch. Ration. Mech. Anal.  180 (2006), 151--182

\bibitem{BT} Braides, A. and Truskinovsky, L.
\textit{Asymptotic expansions by Gamma-convergence}.
 Cont. Mech. Therm. {\bf 20} (2008), 21--62

\bibitem{Ch} Chambolle A. \textit{Un th\'eor\`eme de $\Gamma$-convergence pour la segmentation des signaux}. 
 C. R. Acad. Sci., Paris, Ser. I 314 (1992), 191--196

\bibitem{sei}Dal Maso, G. \textit{An Introduction to $\Gamma$-convergence}. Progress in Nonlinear Differential Equations and Their Applications Vol. 8, Birkhauser, Boston, 1993

\bibitem{GG}Ghisi, M. and Gobbino, M. \textit{A class of local classical solutions for the one-dimensional Perona-Malik equation}. Trans. American Math. Soc. 361 (2009), 6429--6446

\bibitem{sette} Gobbino, M. \textit{Gradient flow for the one-dimensional Mumford-Shah functional}.  Ann. Scuola Norm. Sup. Pisa Cl. Sci. (4) 27 (1999), 145--193

\bibitem{Kic}Kichenassamy, S. \textit{The Perona-Malik paradox}. SIAM J. Appl. Math. 57 
(1997), 1328--1342

\bibitem{MR} Mielke, A. and Roub\'i\v cek, T. 
{\em Rate-Independent Systems.} Springer-Verlag, Berlin, 2015.

\bibitem{otto}Morini, M. and Negri, M.  \textit{Mumford-Shah functional as $\Gamma$-limit of discrete Perona-Malik energies}.  Math. Mod. Meth. Appl. Sci. 13 (2003) 785--805.

\bibitem{MS}Mumford D. and  Shah J. \textit {Optimal approximation by piecewise smooth functions and associated variational problems}. Commun. Pure Appl. Math. 17 (1989), 577--685

\bibitem{PM} Perona P. and Malik, J. \textit {Scale-space and edge detection using anisotropic diffusion}.  IEEE Transactions on pattern analysis and machine intelligence, 12 (1990), 629--639.

\bibitem{PSM}Perona, T. Shiota, and J. Malik. \textit{Anisotropic diffusion}. In {\em Geometry-driven Diffusion in Computer Vision}. Springer Netherlands, 1994. 73-92.

\bibitem{SS} Sandier E. and Serfaty S. \textit{$\Gamma$-convergence of gradient flows
with applications to Ginzburg-Landau}.  Comm. Pure Applied
Math. 57 (2004), 1627--1672.

\bibitem{Zh}Zhang, K. (2006). \textit{Existence of infinitely many solutions for the one-dimensional Perona-Malik model}.  Calc. Var. Partial Diff. Equations, 26(2), 171-199.

\end{thebibliography}
\end{document}